\documentclass[11pt]{amsart}

\usepackage{amsmath,amscd}
\usepackage{amssymb}
\usepackage{amsthm}

\usepackage{graphicx}

\usepackage{mathrsfs}
\usepackage[ocgcolorlinks, linkcolor=blue]{hyperref}

\usepackage{calc}
             {\begin{list}{\arabic{enumi}.}{\usecounter{enumi}%
              \setlength{\labelsep}{0.5em}%
              \settowidth{\labelwidth}{(\arabic{enumi})}%
              \setlength{\leftmargin}{\labelwidth+\labelsep}}}%
             {\end{list}}

\usepackage{comment}

\DeclareRobustCommand{\SkipTocEntry}[5]{}

\usepackage{xcolor}
\definecolor{LOcolor}{RGB}{150,100,0}

\newtheorem{Theorem}{Theorem}[section]

\newtheorem{Lemma}[Theorem]{Lemma}
\newtheorem{Corollary}[Theorem]{Corollary}

\newtheorem{Proposition}[Theorem]{Proposition}

\theoremstyle{definition}
\newtheorem{Definition}[Theorem]{Definition}

\newtheorem{Example}[Theorem]{Example}

\newtheorem{Remark}[Theorem]{Remark}

\numberwithin{equation}{section}

\newcommand{\mR}{\mathbb{R}}                    
\newcommand{\mC}{\mathbb{C}}                    

\newcommand{\ol}[1]{\overline{#1}}

\newcommand{\supp}{\mathrm{supp}}

\newcommand{\p}{\partial}

\newcommand{\gm}{g_{\mathrm{Min}}}
\def\Min{\textrm{Min}}

\def\p{\partial}
\def\R{\mathbb R}
\def\Min{\textrm{Min}}
\DeclareMathOperator{\WF}{WF}

\newcounter{sidenote}

\begin{document}

\title{Rigidity in the Lorentzian Calder\'on problem with formally determined data} 

\author[L. Oksanen]{Lauri Oksanen}
\address{Department of Mathematics and Statistics, University of Helsinki, PO Box 68, 00014 Helsinki, Finland}
\email{lauri.oksanen@helsinki.fi}

\author{Rakesh}
\address{Department of Mathematical Sciences, University of Delaware, Newark, DE 19716, USA}
\email{rakesh@udel.edu}

\author[M. Salo]{Mikko Salo}
\address{Department of Mathematics and Statistics, University of Jyv\"askyl\"a, PO Box 35, 40014 Jyv\"askyl\"a, Finland}
\email{mikko.j.salo@jyu.fi}




\begin{abstract}
We study the Lorentzian Calder\'on problem, where the objective is to determine a globally hyperbolic Lorentzian metric up to a boundary fixing diffeomorphism from boundary measurements given by the hyperbolic Dirichlet-to-Neumann map. This problem is a wave equation analogue of the Calder\'on problem on Riemannian manifolds. We prove that if a globally hyperbolic metric agrees with the Minkowski metric outside a compact set and has the same hyperbolic Dirichlet-to-Neumann map as the Minkowski metric, then it must be the Minkowski metric up to diffeomorphism. In fact we prove the same result with a much smaller amount of measurements, thus solving a formally determined inverse problem. To prove these results we introduce a new method for geometric hyperbolic inverse problems. The method is based on distorted plane wave solutions and on a combination of geometric, topological and unique continuation arguments.
\end{abstract}

\maketitle

\vspace{-10pt}

\section{Introduction} \label{sec_intro}

\subsection{Background} \label{subsec_background}

The Calder\'on problem \cite{calderon1980} is the mathematical model underlying Electrical Impedance Tomography, where the objective is to determine the electrical conductivity of a medium from voltage and current measurements. This is a fundamental inverse boundary value problem for elliptic equations and there is a substantial literature \cite{uhlmann2009}. The case of matrix conductivities has a natural formulation in terms of Riemannian geometry. Let $(M,g)$ be a compact oriented Riemannian manifold with smooth boundary, and consider the Dirichlet problem 
\[
\Delta_g u = 0 \text{ in $M$}, \qquad u|_{\p M} = f,
\]
where $\Delta_g$ is the Laplace-Beltrami operator. The elliptic Dirichlet-to-Neumann map is the operator 
\[
\Lambda_g^{\mathrm{Ell}}: C^{\infty}(\p M) \to C^{\infty}(\p M), \ \ \Lambda_g^{\mathrm{Ell}} f = \p_{\nu} u|_{\p M},
\]
where $\nu$ denotes the unit outer normal vector to $\p M$.

The Calder\'on problem amounts to determining the manifold $M$ and metric $g$ from boundary measurements encoded by $\Lambda_g^{\mathrm{Ell}}$. There is a natural gauge due to coordinate invariance: uniqueness in the Calder\'on problem when $\dim(M) \geq 3$ amounts to showing that if $\Lambda_g^{\mathrm{Ell}} = \Lambda_{g_0}^{\mathrm{Ell}}$, then $g = F^* g_0$ for a diffeomorphism $F: M \to M$ with $F|_{\p M} = \mathrm{id}$. This problem is open. There are positive results when $(M,g)$ is real-analytic \cite{lassas2001} or for certain metrics $g$ having conformal product structure, that is, $g(t,x) = c(t,x)(dt^2 + h(x))$ where $h$ is a Riemannian metric \cite{dos-santos-ferreira2009, dos-santos-ferreira2016, carstea2023}.

The Lorentzian Calder\'on problem is a wave equation analogue of the above problem. Let $(M,g)$ be a Lorentzian manifold with smooth boundary, and let $\Box_g$ be the corresponding wave operator. (See Section \ref{sec_lorentzian} for more on Lorentzian geometry.) We wish to formulate boundary measurements for the equation $\Box_g u = 0$ in $M$. For wave equations it is natural to consider a Cauchy-Dirichlet problem, and a standard condition for the well-posedness of the Cauchy problem is that the metric $g$ is \emph{globally hyperbolic}. For manifolds $M$ with boundary, as in \cite{alexakisfeizmohammadioksanen2022}, well-posedness follows from the existence of a proper, surjective temporal function. A \emph{temporal function} is a smooth function $\tau: M \to \mR$  for which $d\tau$ is timelike, and then $\tau$ can be considered as a global time coordinate. If additionally $\p M$ is timelike, then the problem 
\[
\Box_g u = 0 \text{ in $M$}, \qquad u|_{\p M} = f, \qquad u|_{\{ \tau \ll 0 \}} = 0,
\]
has a unique smooth solution for any $f \in C^{\infty}_c(\p M)$ \cite[Theorem 24.1.1]{hormander}. We can define the hyperbolic Dirichlet-to-Neumann map as the operator 
\[
\Lambda_g^{\mathrm{Hyp}}: C^{\infty}_c(\p M) \to C^{\infty}(\p M), \ \ \Lambda_g f = \p_{\nu} u|_{\p M}.
\]
Uniqueness in the Lorentzian Calder\'on problem, when $\dim(M) \geq 3$, amounts to showing that if $\Lambda_g^{\mathrm{Hyp}} = \Lambda_{g_0}^{\mathrm{Hyp}}$ for two metrics $g$ and $g_0$ satisfying the above conditions, then $F^* g = g_0$ for some diffeomorphism $F: M \to M$ with $F|_{\p M} = \mathrm{id}$.

There is a large literature on the classical case where $(M,g)$ has \emph{ultrastatic} (or product) structure, that is, $M = \mR \times N$ and $g(t,x) = -dt^2 + h(x)$ where $(N,h)$ is a compact Riemannian manifold with smooth boundary. Then the wave operator is $\Box_g = \p_t^2 - \Delta_{h}$, and the question above is equivalent with the Gelfand problem \cite{gelfand1957} or the inverse boundary spectral problem \cite{katchalov2001}. The Boundary Control method, introduced in \cite{belishev1988, belishev1992} and based on the time-optimal unique continuation theorem in \cite{tataru1995}, shows that $\Lambda_g^{\mathrm{Hyp}}$ determines $g$ up to diffeomorphism. There are many further developments, see e.g. \cite{katchalov2001, lassas2014, lassas2018, kurylev2018a}. The unique continuation theorem remains valid for metrics $g(t,x)$ depending real-analytically on $t$, and in this case, a version of the Boundary Control method has been developed in \cite{eskin2007}. However, the Boundary Control method does not extend to general Lorentzian metrics, since the time-optimal unique continuation theorem fails in general \cite{alinhac1983}. It is striking that the obstacles in both the elliptic and hyperbolic Calder\'on problems (i.e.\ absence of real-analyticity or product structure) are similar. The recent works \cite{alexakisfeizmohammadioksanen2022, alexakis2024} make progress in the Lorentzian Calder\'on problem in a fixed conformal class by proving a time-optimal unique continuation theorem where real-analyticity is replaced by Lorentzian curvature bounds.

Another approach to the Lorentzian Calder\'on problem is based on geometrical optics solutions. These are high frequency solutions to $\Box_g u = 0$ that propagate along light rays (=null geodesics), and they allow one to extract information on a Lorentzian scattering relation of $g$ or light ray transforms of lower order terms from $\Lambda_g^{\mathrm{Hyp}}$ \cite{stefanov2018}. However, the light ray transform is only known to be invertible assuming real-analyticity and convex foliations \cite{stefanov2017, mazzucchelli2023} or product structure \cite{feizmohammadi2021a}, and for the Lorentzian scattering relation there are partial results \cite{munozthon2024, stefanov2024b, stefanov2024, uhlmann2021, wang2024}. For related results on the Lorentzian Calder\'on problem see \cite{stefanov2018, feizmohammadi2021b}.

In the ultrastatic case the geometrical optics method reduces the Lorentzian Calder\'on problem to the scattering/lens/boundary rigidity problem, which is another classical inverse problem in Riemannian geometry dating back to \cite{herglotz1907}. A modern geometric formulation was posed in \cite{michel1981}. Among the first contributions, \cite{gromov1983} gave a rigidity result: if $g$ has the same boundary distance function as the Euclidean metric $g_{\mathrm{Eucl}}$, then $g = F^* g_{\mathrm{Eucl}}$ for a boundary fixing diffeomorphism $F$. Much stronger results are now available (see the recent survey \cite{stefanov2019b}). For applications to the Lorentzian Calder\'on problem, see \cite{stefanov2005a, stefanov2016} and references therein.

We now describe the new contributions of this article. We prove an analogue for the Lorentzian Calder\'on problem of the rigidity result in  \cite{gromov1983}: if $g$ is a globally hyperbolic metric in $\mR^{1+n}$ that agrees with the Minkowski metric $\gm$ outside a compact set and satisfies $\Lambda_g^{\mathrm{Hyp}} = \Lambda_{\gm}^{\mathrm{Hyp}}$, then $g = F^* \gm$ for some diffeomorphism with $F = \mathrm{id}$ outside a larger compact set. The main point is that $g$ can be an arbitrary smooth metric in a compact set (no real-analyticity, product type or curvature assumptions) if it is globally hyperbolic. The corresponding rigidity question for the classical Calder\'on problem, that is, does $\Lambda_g^{\mathrm{Ell}} = \Lambda_{g_{\mathrm{Eucl}}}^{\mathrm{Ell}}$ imply that $g$ is isometric to $g_{\mathrm{Eucl}}$, is open. Even the local version, where $g$ is assumed to be close to $g_{\mathrm{Eucl}}$, remains unsolved.

To prove this result we introduce a new method for geometric hyperbolic  inverse problems. The method employs distorted plane waves, which form a useful class of special solutions in inverse hyperbolic problems that are well adapted to reductions to unique continuation problems (compare with geometrical optics solutions, which are related to scattering relation and ray transform problems). Our reduction involves a geometric result (see Theorem \ref{thm_geometric}) that may be interesting in its own right. For our method we do not need time-optimal unique continuation as in the Boundary Control method or as in \cite{alexakisfeizmohammadioksanen2022}. In fact basically any unique continuation result that holds at  large times suffices.

Finally, the method actually requires much less data than the full map $\Lambda_g^{\mathrm{Hyp}}$ whose Schwartz kernel depends on $2n$ variables. We solve a formally determined inverse problem where the metric $g$ depending on $n+1$ variables is determined up to diffeomorphism from measurements depending on $n+1$ variables. The method is somewhat inspired by \cite{rakesh2020a, rakesh2020b, ma2022}, which also use plane waves but adapt the Bukhgeim-Klibanov method  for solving formally determined inverse problems \cite{bukhgeim1981}. See also the related work \cite{romanov2002}. Many further references in this direction are given in the companion article \cite{oksanen2024} that addresses the special case of ultrastatic manifolds.

\subsection{Main results}

We consider a smooth Lorentzian metric $g$ in $\mR^{1+n}$ with signature $(-,+,\ldots,+)$. To simplify the statements we will assume that $n \geq 2$ (see Remark \ref{rmk_onedim} for the case $n=1$ that involves an additional conformal invariance). We assume that $(\mR^{1+n},g)$ is time-oriented and that the space-time structure agrees with the Minkowski one outside $\mR \times \ol{B}$ where $B$ is the unit ball in $\mR^n$. More precisely, considering the Cartesian coordinates $(t,x)$ in $\mR^{1+n}$, we assume that 
\begin{equation} \label{metric_a1}
g=\gm \text{ outside $\mR \times \ol{B}$},
\end{equation}
and choose the time-orientation so that $dt$ is future-directed outside $\mR \times \ol{B}$.

In order to define boundary measurements, we will also need well-posedness of the Cauchy problem for the Lorentzian wave operator $\Box_g$ defined by 
\[
\Box_g u = -|g|^{-1/2} \sum_{j,k=0}^n \p_j(|g|^{1/2} g^{jk} \p_k u).
\]
Here $x_0=t$, and with our sign conventions $\gm = -dt^2 + dx^2$ and $\Box_{\gm} = \p_t^2 - \Delta_x$. We will assume the following (see Definition \ref{def_gh}):
\begin{equation} \label{metric_a2}
\text{$g$ is globally hyperbolic.}
\end{equation}
We recall that a (topological) hypersurface $\Sigma \subset \mR^{1+n}$ is a \emph{Cauchy surface} in $(\mR^{1+n},g)$ if every inextendible timelike curve meets $\Sigma$ exactly once. 
Assumption \eqref{metric_a2} is equivalent with the existence of a smooth spacelike Cauchy surface by Proposition \ref{prop_gh}, and it implies well-posedness of the Cauchy problem for $\Box_g$ with Cauchy data prescribed on any such surface, see Proposition \ref{prop_wp}. Any metric of the form $g = c(t,x)(-dt^2 + h_t(x))$ is globally hyperbolic if $c > 0$ and $h_t$ is a Riemannian metric on $\mR^n$ that has a uniform lower bound and depends smoothly on $t$, see Example \ref{ex_gh} for details and \cite[Section 1.3]{barginouxpfaffle2007} for further examples. On the other hand, if $g$ admits a closed causal curve (causal loop) then $g$ cannot be globally hyperbolic.

Finally we need a unique continuation property for the wave equation. If $r \in \mR$ and $I = (r,\infty)$, we say that $\Box_g$ has the \emph{unique continuation property in $I \times B$} if any $u \in C^{\infty}(I \times \ol{B})$ solving   
\[
\Box_g u = 0 \text{ in $I \times B$}, \qquad u|_{I \times \p B} = \p_{\nu} u|_{I \times \p B} = 0,
\]
must satisfy $u = 0$ in $I_1 \times B$ where $I_1 = (r_1,\infty)$ for some $r_1 > r$. We assume:
\begin{equation} \label{metric_a3}
\left\{ \begin{array}{c}
\text{There is $R_0 > 0$ such that $\Box_g$ has the unique continuation property } \\[3pt]
\text{in $(r,\infty) \times B$ for any $r > R_0$.}
\end{array} \right.
\end{equation}
Examples of metrics on $\mR \times \ol{B}$ satisfying \eqref{metric_a3} include the following:
\begin{itemize}
\item 
Ultrastatic metrics $g = -dt^2 + h(x)$ where $h$ is a Riemannian metric on $\ol{B}$ (for time-optimal unique continuation see \cite{tataru1995}).
\item 
Metrics that admit strictly pseudoconvex functions as in H\"ormander's unique continuation theorem \cite[Theorem 28.3.4]{hormander} whose level sets sweep out an appropriate region in $\mR \times \ol{B}$. If $h$ is a Riemannian metric on $\mR^n$ that admits a strictly convex function $\varphi(x)$, then $\psi = -t^2 + \alpha \varphi(x)$ for suitable $\alpha > 0$ is strictly pseudoconvex for $g = -dt^2 + h(x)$. Another example is given by metrics that satisfy a Lorentzian curvature condition (see \cite[Theorem 2.1]{alexakisfeizmohammadioksanen2022} for time-optimal unique continuation).
\item 
Any smooth compactly supported perturbation of the above examples satisfies \eqref{metric_a3}.
\end{itemize}
There are also counterexamples where unique continuation fails to hold across certain timelike hypersurfaces for wave operators with smooth potentials \cite{alinhac1983, alinhac1995, lerner2019}.

We now describe the boundary measurements in the inverse problem. Let $g$ be a Lorentzian metric on $\mR^{1+n}$ satisfying \eqref{metric_a1}--\eqref{metric_a2}. For any $\omega \in S^{n-1}$ and $s \in \mR$, the Minkowski plane wave $U_{\omega,s}^{\gm}(t,x) = H(t-x \cdot \omega-s)$ solves $\Box_{\gm} U_{\omega,s}^{\gm} = 0$. Here $H(t)$ is the Heaviside function. We define the corresponding (distorted) plane wave $U_{\omega,s} = U_{\omega,s}^g$ for the metric $g$ as the solution of 
\[
\Box_g U_{\omega,s} = 0 \text{ in $\mR^{1+n}$}, \qquad U_{\omega,s}|_{\{\tau \ll 0\}} = H(t-x \cdot \omega-s).
\]
Here $\tau$ is a suitable temporal function that acts as a global time coordinate on a globally hyperbolic manifold, and the above problem is well-posed by global hyperbolicity (for details see Propositions \ref{prop_gh} and \ref{prop_pw}).

The data for our inverse problem is given by the restrictions $U_{\omega,s}|_{\mR \times \p B}$ for all $\omega \in \Omega$ and for all time-delay parameters $s \in \mR$, where  $\Omega$ is the finite set 
\begin{equation} \label{finite_set_def}
\Omega = \{ \pm e_j, \frac{1}{\sqrt{2}}(e_j + e_k) \,:\, 1 \leq j, k \leq n, \ j < k \}.
\end{equation}
Here $e_j \in \mR^n$ is the $j$th coordinate vector.
In other words, we consider the forward operator 
\begin{align*}
\mathcal{F}_{\Omega}: \ \ g \mapsto (\mathcal{F}_{\omega}(g))_{\omega \in \Omega},
\end{align*}
where $\mathcal{F}_{\omega}(g) = (U_{\omega,s}^g|_{\mR \times \p B})_{s \in \mR}$. Note that when \eqref{metric_a1} holds, the measurement $U_{\omega,s}^g|_{\mR \times \p B}$ is well defined as a distribution on $\mR \times \p B$ by wave front set properties.  Indeed, $\WF(U_{\omega,s}^g)$ is contained in the characteristic set of $\Box_g$ by \cite[Theorem 8.3.1]{hormander} and this set is disjoint from $N^*(\mR \times \p B)$ by \eqref{metric_a1}, so $U_{\omega,s}^g|_{\mR \times \p B}$ is well defined by \cite[Corollary 8.2.7]{hormander}.

We will prove the following.

\begin{Theorem} \label{thm_rigidity}
Let $g$ be a smooth Lorentzian metric in $\mR^{1+n}$ satisfying \eqref{metric_a1}--\eqref{metric_a3}, and suppose that 
\[
\mathcal{F}_{\Omega}(g) = \mathcal{F}_{\Omega}(\gm).
\]
Then 
\[
g = F^* \gm,
\]
where $F: \mR^{1+n} \to \mR^{1+n}$ is a diffeomorphism that satisfies $F = \mathrm{id}$ outside $\mR \times \ol{B}$.
\end{Theorem}

If $g$ is Minkowski outside a compact set, then \eqref{metric_a3} is always satisfied and we have the following corollary. By Example \ref{ex_gh} the corollary applies in particular to any metric $g = c(t,x)(-dt^2 + h_t(x))$ with $g = \gm$ outside a compact set.

\begin{Corollary} \label{cor_compact_perturbation}
Let $g$ be a globally hyperbolic Lorentzian metric on $\mR^{1+n}$ that satisfies $g = \gm$ outside $(-T,T) \times \ol{B}$ for some $T > 0$. If $\mathcal{F}_{\Omega}(g) = \mathcal{F}_{\Omega}(\gm)$, then $g = F^* \gm \text{ in $\mR^{1+n}$}$ for some diffeomorphism $F: \mR^{1+n} \to \mR^{1+n}$ with $F = \mathrm{id}$ outside $(-T-2,T+2) \times B$.
\end{Corollary}

The above results solve a formally determined inverse problem, stating that a metric depending on $1+n$ variables can be recovered up to diffeomorphism from the knowledge of plane waves $U_{\omega,s}$ on $\mR \times \p B$ for all time delay variables $s \in \mR$ (and for finitely many directions $\omega$). Thus the boundary measurements also depend on $1+n$ variables.

In the standard Lorentzian Calder\'on problem we assume knowledge of the Cauchy data of all solutions of the wave equation instead of just the plane waves. Any metric $g$ satisfying \eqref{metric_a1}--\eqref{metric_a2} restricts to a metric on $\mR \times \ol{B}$ for which the map $\Lambda_g^{\mathrm{Hyp}}: C^{\infty}_c(\mR \times \p B) \to C^{\infty}(\mR \times \p B)$ is well-defined (see Section \ref{subsec_background} and Proposition \ref{prop_tau_proper}). Thus we also obtain the following result.

\begin{Corollary} \label{cor_dnmap}
Let $g$ be a smooth Lorentzian metric in $\mR^{1+n}$ satisfying \eqref{metric_a1}--\eqref{metric_a3}, and suppose that 
\[
\Lambda_g^{\mathrm{Hyp}} = \Lambda_{\gm}^{\mathrm{Hyp}}.
\]
Then $g = F^* \gm$ for some diffeomorphism $F: \mR^{1+n} \to \mR^{1+n}$ with $F = \mathrm{id}$ outside $\mR \times \ol{B}$.
\end{Corollary}

In particular, Corollary \ref{cor_dnmap} is valid for any globally hyperbolic metric $g$ in $\mR^{1+n}$ with $g = \gm$ outside $(-T,T) \times \ol{B}$ for some $T > 0$. In this case $F = \mathrm{id}$ outside $(-T-2,T+2) \times \ol{B}$.

We remark that in the ultrastatic case where $M = \mR \times N$ and $g = -dt^2 + h$ with $(N,h)$ a compact Riemannian manifold with boundary, we obtain similar results from measurements at a fixed time delay $s \in \mR$. This case is addressed in detail in the companion article \cite{oksanen2024}.

\subsection{Methods}

The proof of Theorem \ref{thm_rigidity} proceeds roughly along the following lines.

\begin{enumerate}
\item[1.] 
From boundary measurements of distorted plane waves, we recover information on a single plane wave (SPW) scattering relation of $g$ in finitely many directions. See Definition \ref{def_spw_scattering_relation}. \\[0pt]
\item[2.] 
By comparing the SPW scattering relations of $g$ and $\gm$, we show using geometric and topological arguments that $g$ has a directional simplicity property (a certain family of null geodesics parametrizes $\mR^{1+n}$ smoothly). \\[0pt]
\item[3.] 
From directional simplicity, we observe that the distorted plane waves are in fact conormal distributions related to solutions of eikonal equations. \\[0pt]
\item[4.] 
We construct the desired diffeomorphism $F$ out of solutions of eikonal equations. At this point we employ the unique continuation property \eqref{metric_a3} and specific properties of $\gm$.
\end{enumerate}

Let us now describe the above steps in more detail. Assume that $g$ satisfies \eqref{metric_a1}, fix $\omega \in S^{n-1}$ and write 
\[
\Sigma_{\pm} = \{ (t,x) \in \mR^{1+n} \,:\, x \cdot \omega = \pm 1 \}, \qquad \Sigma_{\pm,\sigma} = \Sigma_{\pm} \cap \{ t = \sigma \}.
\]
Given $z \in \Sigma_-$, let $\gamma_z(r)$ be the $g$-geodesic satisfying $\gamma_z(0) = z$ and $\dot{\gamma}_z(0) = (1,\omega)$ that is defined in the maximal interval $(-\infty,\rho(z))$ where $\rho(z) \in \mR_+ \cup \{ \infty \}$. These geodesics are closely connected to distorted plane waves, since by Proposition \ref{prop_pw} the wave front set of $U_{\omega,s}$ lies over the union of the curves $\gamma_z$ for $z \in \Sigma_{-,s-1}$.

\begin{Definition} \label{def_spw_scattering_relation}
For any $z \in \Sigma_-$ let $r_+(z) = \sup \,\{ r < \rho(z) \,:\, \gamma_z([0,r]) \subset \{ x \cdot \omega \leq 1 \} \}$ be the time when $\gamma_z$ exits $\{ x \cdot \omega \leq 1 \}$, and let $\mathcal{T}_g = \{ z \in \Sigma_- \,:\, r_+(z) = \rho(z) \}$ be the trapped set. Given any positive $\lambda \in C^{\infty}(\Sigma_+)$, the \emph{single plane wave (SPW) scattering relation} is the map 
\[
\beta_g^{\lambda}: \Sigma_- \setminus \mathcal{T}_g \to T^* \Sigma_+, \ z \mapsto (\gamma_z(r_+(z)), \lambda(\gamma_z(r_+(z))) \dot{\gamma}_z(r_+(z))).
\]
We write $\beta_g = \beta_g^1$ when $\lambda \equiv 1$.
\end{Definition}

Thus the SPW scattering relation maps the start point on $\Sigma_-$ of a geodesic $\gamma_z$ going in direction $(1,\omega)$ to its end point and direction at the time when the geodesic exits $\{ x \cdot \omega \leq 1 \}$. Moreover, the end direction may be multiplied by a positive scaling factor $\lambda$. The set $\beta_g(\Sigma_{-,s-1} \setminus \mathcal{T}_g)$ is related to the wave front set of the plane wave $U_{\omega,s}$. Our first result shows that boundary measurements of plane waves determine the SPW scattering relation as a set.

\begin{Theorem} \label{thm_pw_spw_scattering}
Let $g$ be a Lorentzian metric in $\mR^{1+n}$ satisfying \eqref{metric_a1}--\eqref{metric_a2}, let $\omega \in S^{n-1}$, and suppose that 
\[
\mathcal{F}_{\omega}(g) = \mathcal{F}_{\omega}(\gm).
\]
Then there is a positive $\lambda \in C^{\infty}(\Sigma_+)$ such that, for any $\sigma \in \mR$, 
\[
\beta_g(\Sigma_{-,\sigma} \setminus \mathcal{T}_g) = \beta_{\gm}^{\lambda}(\Sigma_{-,\sigma}).
\]
\end{Theorem}

The conclusion means in particular that any point $(z_+, \lambda(1,\omega))$, with $z_+ \in \Sigma_{+,\sigma+2}$, is on a geodesic starting at $(z,(1,\omega))$ for some $z \in \Sigma_{-,\sigma}$, but we do not know which point $z$ it came from. At this point the scaling factor $\lambda$ is undetermined, but later we will show that in fact $\lambda \equiv 1$.

The next important step in the proof of Theorem \ref{thm_rigidity} is the following geometric result that can be understood as a directional simplicity result for the Lorentzian metric $g$. Recall that a Riemannian manifold $(M,h)$ is simple if the exponential map at any point $p \in M$ is a diffeomorphism onto $M$, and $M$ is strictly convex and simply connected. The first property means that the radial geodesics emanating from $p$ parametrize $M$ smoothly. The following result states instead that null geodesics emanating from $\Sigma_-$  in direction $(1,\omega)$ parametrize $\mR^{1+n}$ smoothly.

\begin{Theorem} \label{thm_geometric}
Let $g$ be a Lorentzian metric in $\mR^{1+n}$ satisfying \eqref{metric_a1}--\eqref{metric_a2}, let $\omega \in S^{n-1}$, and suppose that for any $\sigma \in \mR$ we have 
\[
\beta_g(\Sigma_{-,\sigma} \setminus \mathcal{T}_g) = \beta_{\gm}^{\lambda}(\Sigma_{-,\sigma})
\]
for some positive $\lambda \in C^{\infty}(\Sigma_+)$. Then $\mathcal{T}_g = \emptyset$, each $\gamma_z$ is defined on all of $\mR$, and the map 
\[
\Phi: \Sigma_- \times \mR \to \mR^{1+n}, \ \ \Phi(z,r) = \gamma_z(r)
\]
is a diffeomorphism. Moreover, $\dot{\gamma}_z(r) = \hat{\lambda}(\gamma_z(r))(1,\omega)$ when $\gamma_z(r) \notin \mR \times B$, where $\hat{\lambda} \in C^{\infty}(\mR^{1+n})$ is positive and satisfies $\hat{\lambda}|_{\{ x \cdot \omega \leq -1\}} = 1$.
\end{Theorem}

The result above is an analogue of the result in boundary rigidity stating that the simplicity of a Riemannian metric can be read off from its boundary distance function \cite[Section 3.8]{paternain2023}. For simple manifolds, the knowledge of the boundary distance function is equivalent to the knowledge of the full scattering relation \cite[Section 11.3]{paternain2023}. In contrast, Theorem \ref{thm_geometric} only assumes partial knowledge of the scattering relation (related to null geodesics with initial direction $(1,\omega)$ that is normal to $\Sigma_{-,\sigma}$). The assumption \eqref{metric_a1} is crucial since a local version of Theorem \ref{thm_geometric} may fail, see Figure \ref{fig_cex}.

The proof of Theorem \ref{thm_geometric} begins with the nontrapping property $\mathcal{T}_g = \emptyset$. There the key fact is that any point $w \in \Sigma_+$ lies on some geodesic $\gamma_z$ with $z \in \Sigma_-$, and one obtains a map $\Sigma_+ \to \Sigma_-$, $w \mapsto z$ that turns out to be a diffeomorphism by a topological argument. The next step is to show that the set $J^-(\Sigma_{+,\sigma}) \cap \Sigma_-$, where $J^-$ denotes causal past, agrees with the corresponding set for $\gm$, that is, it is equal to $\{ t \leq \sigma - 2 \} \cap \Sigma_-$. This is based on showing that $\p J^{-}(\Sigma_{+,\sigma})$ is covered by geodesics $\gamma_z$ that are null and normal to $\Sigma_{+,\sigma}$. Finally we show that the derivative of $\Phi$ cannot be singular anywhere, since otherwise one would obtain a focal point of $\Sigma_{-,\sigma-2}$ along some $\gamma_z$ and hence there would be a timelike curve from $\Sigma_{-,\sigma-2}$ to $\Sigma_{+,\sigma}$, which would contradict the earlier statement on  $J^{-}(\Sigma_{+,\sigma})$. The proof of Theorem \ref{thm_geometric} is concluded by invoking the Hadamard global inverse function theorem, which involves showing that $\Phi$ is proper.

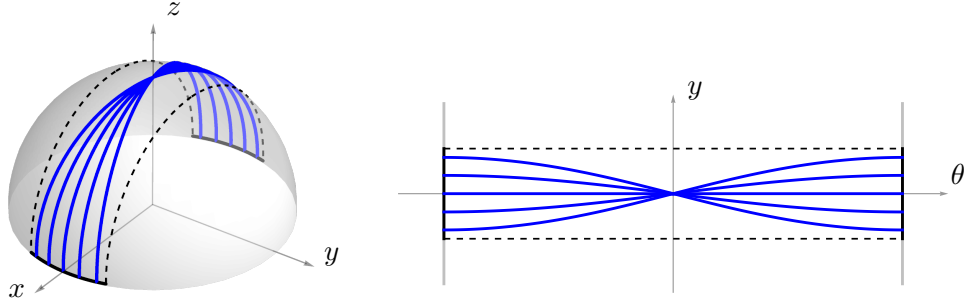
\begin{figure}
\centering
\begin{picture}(5cm,4cm)
\put(0,0){\includegraphics[scale=0.6,trim={4.5cm 2.5cm 3.5cm 3cm},clip]{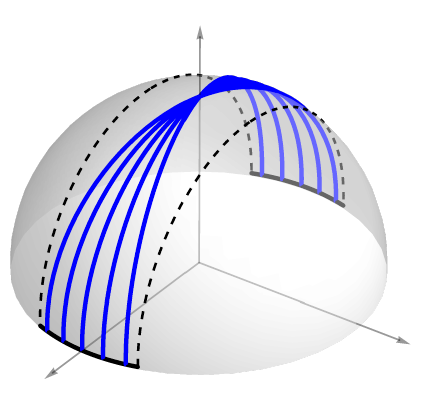}}
\put(0.1cm,0.1cm){$x$}
\put(4.3cm,0.6cm){$y$}
\put(2.2cm,3.85cm){$z$}
\end{picture}
\begin{picture}(7cm,4cm)
\put(0,0){\includegraphics[scale=0.6]{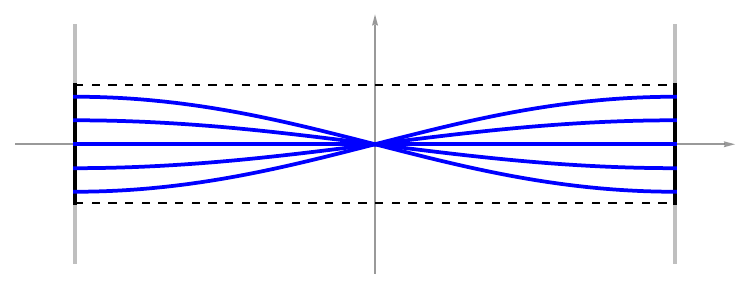}}
\put(4cm,2.8cm){$y$}
\put(7.5cm,1.6cm){$\theta$}
\end{picture}
\caption{The points $p = (x,y,z)$ in the strip on the sphere are mapped 
to the rectangle in the plane by $p \mapsto (\theta, y)$ where $\theta$
is the (signed) angle with the $yz$-plane. 
The pushforward metric under this map can be extended to a smooth Riemannian metric $h$ on the plane. 
Blue $h$-geodesics are normal to the two gray lines $\theta = \pm \pi/2$. The conclusions of Theorem \ref{thm_geometric} do not hold for the metric $g = -dt^2 + h$. 
This ultrastatic metric satisfies \eqref{metric_a2} but $g \ne \gm$ near the boundary of the rectangle, highlighting the importance of \eqref{metric_a1}.
}
\label{fig_cex}
\end{figure}

The next step is to show that the directional simplicity property in Theorem \ref{thm_geometric} implies that the plane waves $U_{\omega,s}$ are conormal distributions that can be represented in terms of solutions of eikonal equations.

\begin{Theorem} \label{thm_eikonal_pw}
Let $g$ satisfy \eqref{metric_a1}--\eqref{metric_a2}, let $\omega \in S^{n-1}$ and suppose that the conclusions of Theorem \ref{thm_geometric} hold. Then there is a smooth function $\varphi_{\omega}$ in $\mR^{1+n}$ solving the eikonal equation 
\[
g(d\varphi_{\omega}, d\varphi_{\omega}) = 0 \text{ in $\mR^{1+n}$}
\]
such that $\varphi_{\omega} = t - x \cdot \omega$ outside $\mR \times \ol{B}$ and $d\varphi_{\omega}$ is nowhere vanishing. Moreover, $U_{\omega,s}$ is a conormal distribution having the representation 
\[
U_{\omega,s} = u_{\omega,s} H(\varphi_{\omega}-s) \text{ in $\mR^{1+n}$},
\]
where $u_{\omega,s}$ is smooth.
\end{Theorem}

The final step in the proof is to use the solutions $\varphi_{\omega}$ of the eikonal equation and the unique continuation property \eqref{metric_a3} to construct a diffeomorphism $F: \mR^{1+n} \to \mR^{1+n}$ such that $F^* g = \gm$ and $F = \mathrm{id}$ outside $\mR \times \ol{B}$. The assumption in Theorem \ref{thm_rigidity} is $\mathcal{F}_{\omega}(g) = \mathcal{F}_{\omega}(\gm)$ for $\omega$ in the finite set \eqref{finite_set_def}. We define the functions 
\begin{align*}
\sigma &= \frac{1}{2} (\varphi_{e_1} + \varphi_{-e_1}), \\
\alpha_j &= \frac{1}{2} (\varphi_{-e_j} - \varphi_{e_j}),
\end{align*}
where $1 \leq j \leq n$. The diffeomorphism $F$ will be given by 
\[
F(t,x) = (\sigma(t,x), \alpha_1(t,x), \ldots, \alpha_n(t,x)).
\]
Since $\varphi_{\omega} = t - x \cdot \omega$ outside $\mR \times \ol{B}$, one has $\sigma = t$ and $\alpha_j = x_j$ outside $\mR \times \ol{B}$. Thus $\sigma$ could be considered as a new time coordinate, and $\alpha_j$ as an analogue of the $x_j$ coordinate.

It remains to show that $F$ is indeed a diffeomorphism and satisfies $F^* g = \gm$. The first key observation is that the plane waves $U_{\omega,s} = u_{\omega,s} H(\varphi_{\omega}-s)$ must satisfy $u_{\omega,s} \equiv 1$ in $\{ \varphi_{\omega} \geq s \}$. In fact, since any constant function vanishes under $\Box_g$, the function $v = u_{\omega,s}-1$ solves the equation 
\[
\Box_g v = 0 \text{ in $\{ \varphi_{\omega} > s \} \cap (\mR \times B)$}.
\]
The assumption $\mathcal{F}_{\omega}(g) = \mathcal{F}_{\omega}(\gm)$ implies that $v$ has zero Cauchy data on $\{ \varphi_{\omega} > s \} \cap (\mR \times \p B)$. We can now invoke the unique continuation property \eqref{metric_a3}, at some large enough time, and uniqueness for the backward Cauchy problem to show that $v \equiv 0$ and therefore $u_{\omega,s} \equiv 1$ in $\{ \varphi_{\omega} \geq s \}$.

Now inserting $U_{\omega,s} = H(\varphi_{\omega}-s)$ in the equation $\Box_g U_{\omega,s} = 0$ yields that 
\[
\Box_g \varphi_{\omega} = 0 \text{ on $\{ \varphi_{\omega} = s \}$}.
\]
Since we have access to boundary measurements of plane waves for any time delay $s$, the above holds for any $s$. Consequently $\varphi_{\omega}$ solves the wave equation $\Box_g \varphi_{\omega} = 0$ in addition to solving the eikonal equation. This imposes very strong constraints on the functions $\varphi_{\omega}$. In particular, by the unique continuation property again, for any $\omega \in \Omega$ one obtains the global representation 
\[
\varphi_{\omega} = \sigma - \sum_{j=1}^n \omega_j \alpha_j.
\]
This generalizes the fact that for $\gm$ one has $\varphi_{\omega} = t - x \cdot \omega$ everywhere. The previous property is sufficient for proving that $F^* g = \kappa \gm$ for some positive smooth function $\kappa$. In particular $F$ is a local diffeomorphism. The Hadamard global inverse function theorem, together with some topological properties of $\varphi_{\omega}$, implies that $F$ is a diffeomorphism. An additional argument yields $\kappa \equiv 1$, which concludes the proof of Theorem \ref{thm_rigidity}.

\subsection*{Organization}

This article is organized as follows. Section \ref{sec_intro} is the introduction. In Section \ref{sec_lorentzian} we collect some facts from Lorentzian geometry and discuss certain properties of Cauchy temporal functions. Section \ref{sec_geometric} contains geometric arguments related to the diffeomorphism property and single plane wave scattering relation and gives the proof of Theorem \ref{thm_geometric}. Section \ref{sec_eikonal} relates the diffeomorphism property appearing in Theorem \ref{thm_geometric} with solutions of eikonal equations. Section \ref{sec_pw} establishes the basic properties of plane wave solutions. Our main results for inverse problems are then proved in Section \ref{sec_rigidity_proofs}. Appendix \ref{sec_app} contains certain auxiliary arguments.

\subsection*{Acknowledgements}

L.O.\ was supported by the European Research Council of the European Union, grant 101086697 (LoCal). L.O.\ and M.S.\ were partly supported by the Research Council of Finland, grants 353091 and 353096  (Centre of Excellence in Inverse Modelling and Imaging), 359182 and 359208 (FAME Flagship) as well as 347715. Rakesh’s work was partly funded by grants DMS 1908391 and DMS 2307800 from the National Science Foundation of USA. Views and opinions expressed are those of the authors only and do not necessarily reflect those of the European Union or the other funding organizations. Neither the European Union nor the other funding organizations can be held responsible for them.


\section{Lorentzian geometry} \label{sec_lorentzian}

\subsection{Globally hyperbolic spacetimes}

For basic  facts in Lorentzian geometry see \cite{oneill1983, ringstrom2009}. Let $(M,g)$ be a connected Lorentzian manifold of dimension $1+n$ with signature $(-,+,\dots,+)$. 
We recall that a vector $v \in T_p M \setminus \{0\}$, $p \in M$, is timelike, spacelike or null, if respectively 
    \begin{align}
g(v,v) < 0, \quad g(v,v) > 0, \quad g(v,v) = 0.
    \end{align}
A vector $v$ is said to be causal if it is timelike or null, that is, \ $g(v,v) \leq 0$.
We will assume that $(M,g)$ is time-oriented, that is, there is a vector field $Z$ on $M$ such that $Z(p)$ is timelike for all $p \in M$.
A causal vector $v \in T_p M$, $p \in M$, is future-directed if 
$g(v, Z) < 0$ and past-directed if the opposite inequality holds.

A submanifold $S$ of $M$ is called spacelike if all the non-zero tangent vectors of $S$ are spacelike; this is 
equivalent to the condition that the restriction of $g$ on $S$ is a Riemannian metric on $S$. 

We recall a basic property (see \cite[Exercise 3 on p.\ 155]{oneill1983}).

\begin{Lemma}  \label{lemma_timecone}
Any two future-directed causal vectors $v, w \in T_p M$ satisfy $g(v,w) \leq 0$ with equality if and only if $v$ is null and $w = \lambda v$ for some $\lambda > 0$.
\end{Lemma}

The metric $g$ induces isomorphisms $T_p M \to T_p^* M$, $v \mapsto v^{\flat}$ and $T_p^* M \to T_p M$, $\xi \mapsto \xi^{\sharp}$ via 
\[
v^{\flat}(w) = g(v,w), \qquad \xi(w) = g(\xi^{\sharp}, w).
\]
The inner product of $\xi, \eta \in T_p^* M$ is defined via $g(\xi,\eta) = g(\xi^{\sharp}, \eta^{\sharp})$. We say that $\xi \in T_p^* M$ is timelike if $\xi^{\sharp}$ is, and that $\xi$ is future-directed if $\xi^{\sharp}$ is past-directed. We also write $\mathrm{grad}_g(f) = df^{\sharp}$.

A smooth curve $\gamma : [a,b] \to M$ is timelike if its tangent vector $\dot \gamma(s)$ is timelike for all $s \in [a,b]$. Spacelike, null, causal, as well as future and past-directed curves are defined analogously. These definitions can be extended straightforwardly to piecewise smooth curves. However, a causal (or timelike or null) curve is always required to be either future or past-directed, that is, the tangent vector cannot switch causal cones at the breaking points.

For a pair of points $p, q \in M$ we write $p \le q$ if $p = q$ or there is a future-directed causal curve from $p$ to $q$.
The causal future of a set $S \subset M$ is defined by
    \begin{align}
J^+(S) = \{ q \in M \,:\, \text{there is $p \in S$ satisfying $p \le q$}\}.
    \end{align} 
The past $J^-(S)$ is defined analogously, with the inequality reversed. 
We write $J^\pm(p) = J^\pm(\{p\})$ for points $p \in M$.

\begin{Definition} \label{def_gh}
The manifold $(M,g)$ is \emph{globally hyperbolic}, see \cite{hounnonkpe2019}, if 
\begin{itemize}
\item[(i)] there is no closed causal curve, and 
\item[(ii)] the causal diamond $J^+(p) \cap J^-(q)$ is compact for all $p, q  \in M$.
\end{itemize}
\end{Definition}

If $M$ is non-compact and $1+n \geq 3$, then condition (i) follows from (ii) \cite{hounnonkpe2019}.

A (topological) hypersurface $\Sigma \subset M$ is called a \emph{Cauchy surface} if every inextendible timelike curve meets $\Sigma$ exactly once. We quote a fundamental result.

\begin{Proposition} \label{prop_gh}
The following are equivalent conditions for $(M,g)$.
\begin{enumerate}
\item[(a)]
$(M,g)$ is globally hyperbolic.
\item[(b)] 
There is a smooth spacelike Cauchy surface.
\item[(c)]
There is a \emph{Cauchy temporal function}, that is, a smooth function $\tau: M \to \mR$ that satisfies $g(d\tau, d\tau) < 0$, $d\tau$ is future-directed, and all level sets $\Sigma_r = \tau^{-1}(r)$ are smooth spacelike Cauchy surfaces.
\end{enumerate}
The function $\tau$ in \emph{(c)} can be chosen so that it is surjective, and its restriction to any inextendible future-directed causal curve is strictly increasing and surjective onto $\mR$. 
Moreover, each $\Sigma_r$ is connected, and there is a diffeomorphism $F: \mR \times \Sigma_0 \to M$ obtained by flowing $\Sigma_0$ along $-\mathrm{grad}_g(\tau)$ such that 
    \begin{align}\label{gh_split}
(F^* g)(r,x) = c(r,x)(-dr^2 + h_r(x))
    \end{align}
where $c > 0$ is smooth, $h_r(x)$ is a Riemannian metric on $\Sigma_0$ smoothly depending on $r$, and $r = \tau$. The map $F$ induces a diffeomorphism between $\Sigma_r$ and $\Sigma_{r'}$ for any $r, r' \in \mR$.
\end{Proposition}
\begin{proof}
The equivalence of (a)--(c) was established in \cite{bernal2003, bernal2005} (see also the survey \cite{sanchez2022} or \cite[Chapter 11]{ringstrom2009}). The properties of $\tau$ on causal curves are stated in \cite[Theorem 11.18]{ringstrom2009}, and the diffeomorphism $F$ is given in \cite{bernal2005}. We remark that the connectedness of $\Sigma_0$ follows from the connectedness of $M$: if $x_1, x_2 \in \Sigma_0$, then there is a continuous curve $\gamma(s)$ joining $F(0,x_1)$ and $F(0,x_2)$ in $\mR^{1+n}$, and $F^{-1}(\gamma(s)) = (r(s), x(s))$ where $x(s)$ is a continuous curve in $\Sigma_0$ joining $x_1$ and $x_2$.
\end{proof}

From now on, $\tau$ will denote a Cauchy temporal function that satisfies all the properties in Proposition \ref{prop_gh}. In particular, the Cauchy surfaces $\Sigma_r$ will be smooth, connected and spacelike.

\begin{Example} \label{ex_gh}
Let $g = c(t,x)(-dt^2 + h_t)$ where $c$ is a smooth positive function in $\mR^{1+n}$ and $h_t$ is a Riemannian metric on $\mR^n$ depending smoothly on $t$ such that, for some $\lambda > 0$, 
\[
h_t(v,v) \geq \lambda |v|_{\mathrm{Eucl}}^2
\]
uniformly over $(t,x) \in \mR^{1+n}$ and $v \in \mR^n$. The last condition means that the time cones for $g$ cannot become arbitrarily wide. We claim that $(\mR^{1+n}, g)$ is globally hyperbolic. (Without the lower bound on $h_t$ this may fail, see \cite[Section 6.2]{sanchez2022}.) To see this, let $\gamma$ be a smooth inextendible timelike curve, so in particular $\dot{\gamma}(s) \neq 0$ for all $s$ (the piecewise smooth case is analogous). By a reparametrization we may arrange that $|\dot{\gamma}(s)|_{\mathrm{Eucl}} \equiv 1$, and since $\gamma$ is inextendible it must be defined for all $s \in \mR$. We may assume that $\gamma$ is future-directed (the past-directed case is analogous). Since $\gamma(s) = (\gamma_0(s), \eta(s)) \in \mR \times \mR^n$ is timelike, we have 
\[
\dot{\gamma}_0(s)^2 > h_t(\dot{\eta}(s), \dot{\eta}(s)) \geq \lambda |\dot{\eta}(s)|_{\mathrm{Eucl}}^2 = \lambda(1-\dot{\gamma}_0(s)^2).
\]
This gives 
\[
\dot{\gamma}_0(s) > \sqrt{\frac{\lambda}{1+\lambda}}.
\]
It follows that $s \mapsto \gamma_0(s)$ is bijective $\mR \to \mR$ and $\gamma(s)$ intersects each $\Sigma_r = \{r\} \times \mR^n$ exactly once. Thus any $\Sigma_r$ is a Cauchy surface, and $(\mR^{1+n},g)$ is globally hyperbolic.
\end{Example}

Global hyperbolicity ensures that the Cauchy problem for the wave equation is well-posed and one has finite propagation speed.

\begin{Proposition} \label{prop_wp}
Let $(M,g)$ be globally hyperbolic, let $\tau$ be a Cauchy temporal function and let $\tau_0, s \in \mR$. For any $f \in H^s_{\mathrm{loc}}(M)$ with $\supp(f) \subset \{ \tau \geq \tau_0 \}$, there is a unique $u \in H^{s+1}_{\mathrm{loc}}(M)$ such that 
\[
\Box_g u = f \text{ in $M$}, \qquad u|_{\{ \tau < \tau_0 \}} = 0.
\]
One has finite propagation speed, meaning that 
\[
\supp(u) \subset J^+(\supp(f)).
\]
Similarly, if $\supp(f) \subset \{ \tau \leq \tau_0 \}$, there is a unique $v \in H^{s+1}_{\mathrm{loc}}(M)$ such that 
\[
\Box_g v = f \text{ in $M$}, \qquad v|_{\{ \tau > \tau_0 \}} = 0,
\]
and one has $\supp(v) \subset J^{-}(\supp(f))$.
\end{Proposition}
\begin{proof}
For a distribution $f$ with $\supp(f) \subset \{ \tau \geq \tau_0 \}$ the existence of a distributional solution $u$ with $\supp(u) \subset J^+(\supp(f))$ is given in \cite[Lemma 4.1]{bar2015} (for the smooth case see \cite[Corollary 3.4.3]{barginouxpfaffle2007}). Uniqueness follows from \cite[Theorem 3.1.1]{barginouxpfaffle2007}.

For regularity, if  $f \in H^s_{\mathrm{loc}}(M)$ and $u$ is the distributional solution given above, then $u$ is microlocally $H^{s+2}$ away from the characteristic set of $\Box_g$ by \cite[Theorem 18.1.31]{hormander}. Moreover, by propagation of $H^{s+1}$ singularities \cite[Theorem 26.1.4]{hormander} we see that $u$ is microlocally $H^{s+1}$ along every null bicharacteristic that reaches the set $\{ \tau < \tau_0 \}$ where $u$ vanishes. However, the projection of a null bicharacteristic is a null geodesic by Lemma \ref{lemma_bichar_geodesic}, hence it is a causal curve that meets the Cauchy surface $\{ \tau = \tau_0 \}$ exactly once. Thus every null bicharacteristic reaches $\{ \tau < \tau_0 \}$. It follows that $u$ is microlocally $H^{s+1}$ everywhere in $T^* M \setminus 0$, which gives that $u$ is in $H^{s+1}_{\mathrm{loc}}(M)$ by \cite[Theorem 18.1.31]{hormander}.
\end{proof}

\subsection{Spacetimes that are Minkowski outside a cylinder}

We now consider globally hyperbolic spacetimes $(\mR^{1+n},g)$ that satisfy $g = \gm$ outside $\mR \times \ol{B}$. The next result shows that this condition forces certain relations between $t$ and $\tau$.

\begin{Proposition} \label{prop_tau_proper}
Let $g$ satisfy \eqref{metric_a1}--\eqref{metric_a2}, let $\tau$ be a Cauchy temporal function and let $(t_j, x_j)$ be a sequence in $\mR \times \ol{B}$.
\begin{enumerate}
\item[(a)] 
If $\tau(t_j,x_j) \to \pm \infty$, then $t_j \to \pm \infty$. Moreover, for any $t_0 \in \mR$ there are $r_{\pm} \in \mR$ such that 
\begin{align*}
\{ \tau \leq r_{-} \} &\subset \{ (t,x) \in \mR^{1+n} \,:\, t < t_0 + \max( |x| -1, 0 ) \}, \\
\{ \tau \geq r_{+} \} &\subset \{ (t,x) \in \mR^{1+n} \,:\, t > t_0 - \max( |x| -1, 0 ) \}.
\end{align*}
\item[(b)] 
If $t_j \to \pm \infty$, then $\tau(t_j,x_j) \to \pm \infty$. In particular $\tau|_{\mR \times \ol{B}}$ is a proper map.
\end{enumerate}
\end{Proposition}

We first prove part (a), which will be used frequently in this article. Part (b) will be used to show that the map $\Lambda_g^{\mathrm{Hyp}}$ in Corollary \ref{cor_dnmap} is well defined.

The next lemma describes the set $\Sigma_r =\tau^{-1}(r)$ outside $\mR^n \times \ol{B}$.

\begin{Lemma} \label{lemma_taur_ext}
Let $g$ satisfy \eqref{metric_a1}--\eqref{metric_a2}. There is a smooth real valued function $\alpha$ such that
\[
\Sigma_r \cap \{ |x| \geq 1 \} = \{ (\alpha(r,x), x) \,:\, |x| \geq 1 \}.
\]
For fixed $x$ with $|x| \geq 1$, the map $r \mapsto \alpha(r,x)$ is strictly increasing and maps onto $\mR$.
\end{Lemma}
\begin{proof}
If $|x| \geq 1$, by \eqref{metric_a1} the curve $t \mapsto (t,x)$ is timelike and hence it intersects the Cauchy surface $\Sigma_r$ at a unique time $t = \alpha(r,x)$. Thus $\Sigma_r \cap \{ |x| \geq 1 \} = \{ (\alpha(r,x), x) \,:\, |x| \geq 1 \}$ and 
\begin{equation} \label{tau_alpha_identity}
\tau(\alpha(r,x), x) = r \quad \text{when $|x| \geq 1$}.
\end{equation}
Consider the smooth map $F(t,r,x) = \tau(t, x) - r$. Now $t = \alpha(r,x)$ for $|x| \geq 1$ is the unique time for which $F(t,r,x) = 0$, and $\p_t F = \p_t \tau > 0$  since $t \mapsto (t,x)$ is timelike. The implicit function theorem ensures that there is a smooth local extension of $\alpha$ near any $(r,x)$ with $|x| \geq 1$. Differentiating \eqref{tau_alpha_identity} in $r$ gives that $\p_r \alpha > 0$, and the identity $\alpha(\tau(t,x), x) = t$ shows that $r \mapsto \alpha(r,x)$ maps onto $\mR$.
\end{proof}

\begin{proof}[Proof of Proposition \ref{prop_tau_proper} {\rm (a)}]
We argue by contradiction and suppose that $(t_j, x_j) \in \mR \times \ol{B}$ with $\tau(t_j,x_j) \to -\infty$, but $t_j \not\to - \infty$ (the proof for $+\infty$ is analogous). Then there is $t_0 \in \mR$ so that, after passing to a subsequence, for any $j$ one has $\tau(t_j,x_j) = r_j \leq -j$ and $t_j \geq t_0$.

On the other hand, by Lemma \ref{lemma_taur_ext} 
\begin{equation} \label{sigmar_simple_structure}
\Sigma_r \cap \{ |x| \geq 1 \} = \{ (\alpha(r,x), x) \,:\, |x| \geq 1 \}
\end{equation}
where $\alpha$ is smooth on $\{ |x| \geq 1 \}$ and $r \mapsto \alpha(r,x)$ is strictly increasing and surjective onto $\mR$. Thus for any $x$ with $|x| = 1$, there is $r_x$ such that $\alpha(r,x) \leq t_0-2$ when $r \leq r_x$. By continuity and compactness there is $r_0$ such that $\alpha(r,x) \leq t_0-1$ whenever $|x| = 1$ and $r \leq r_0$.

Now if $j$ is sufficiently large, we have $r_j \leq r_0$ and thus $\Sigma_{r_j} \cap \{ |x| = 1 \} \subset \{ t \leq t_0-1 \}$. On the other hand $(t_j, x_j) \in \Sigma_{r_j}$ and $t_j \geq t_0$. Since $\Sigma_{r_j}$ is connected there is a continuous curve $\gamma(s)$ in $\Sigma_{r_j}$ with $\gamma(0) = (t_j,x_j)$ and $\gamma(1) \in \Sigma_{r_j} \cap \{ |x| = 1 \}$. We can arrange that this curve lies in $\Sigma_{r_j} \cap \{ |x| \leq 1 \}$ simply by replacing it with $\gamma|_{[0,a]}$ where $a$ is the first time when $\gamma$ meets $\{ |x| = 1 \}$. This implies that $\Sigma_{r_j}$ contains a point $(t_0,y_j)$ where $|y_j| \leq 1$. By passing to a subsequence we may assume that $y_j \to y$ where $|y| \leq 1$. But now we have 
\[
\tau(t_0,y_j) = r_j \leq -j, \qquad (t_0,y_j) \to (t_0,y).
\]
Thus $\tau(t_0,y_j)$ would converge both to $-\infty$ and $\tau(t_0,y)$, which is a contradiction.

For the second statement, we first observe that there is $r_-$ so that $\{ \tau \leq r_- \} \cap (\mR \times \ol{B}) \subset \{ t < t_0 \}$ (otherwise one would get a contradiction with the first statement). If $\alpha$ is as above, then for $r \leq r_-$ we have 
\[
\{ \tau = r \} \cap (\mR \times \p B) = \{ (\alpha(r,x), x) \,:\, |x| = 1 \} \subset \{ t < t_0 \}.
\]
Since $g = \gm$ outside $\mR \times \ol{B}$ and since $\{ \tau = r \}$ is spacelike, for $r \leq r_-$ we must have 
\[
\{ \tau = r \} \subset \{ (t,x) \in \mR^{1+n} \,:\, t < t_0 + \max( |x| -1, 0 ) \}.
\]
This proves the second statement.
\end{proof}

To prove Proposition \ref{prop_tau_proper} (b), we state a lemma that is of independent interest. Note that in this lemma $\{ t = t_0 \}$ is not necessarily a Cauchy surface, see Figure \ref{fig_twirl}.

\begin{Lemma} \label{lemma_timelike_intersection}
Let $g$ satisfy \eqref{metric_a1}--\eqref{metric_a2}. Given $t_0 \in \mR$, any inextendible timelike curve meets the set $\{ t = t_0 \}$.
\end{Lemma}
\begin{proof}
Assume for notational simplicity that $t_0 = 0$ and let $\gamma$ be an inextendible future-directed timelike curve. We argue by contradiction and suppose that $\gamma$ stays forever in $\{ t > 0 \}$ (the case with $\{ t < 0 \}$ is analogous). By Proposition \ref{prop_tau_proper} (a), there is $r$ such that $\Sigma_r \subset \{ t <  \max(|x|-1,0) \}$. Since $\Sigma_r$ is a Cauchy surface, $\gamma$ meets $\Sigma_r$ at some $\gamma(s_0) = (t_0,x_0)$. Since $\gamma$ stays in $\{ t > 0 \}$, we must have 
\[
|x_0| > 1 \quad \text{ and } 0 < t_0 < |x_0|-1.
\]
Now $\gamma(s)$ must stay in $J^-(t_0,x_0) \cap \{ t > 0 \}$ for $s \leq s_0$, but by \eqref{metric_a1} this set is contained in the compact set $J^-_{\mathrm{Min}}(t_0,x_0) \cap \{ t \geq 0 \}$ where $J^-_{\mathrm{Min}}$ is the causal past with respect to $\gm$. However, $\tau$ is bounded in any compact set, which contradicts the fact that $\tau(\gamma(s)) \to -\infty$ as $s \to -\infty$.
\end{proof}

\begin{figure}
\centering
\includegraphics[width=0.4\textwidth]{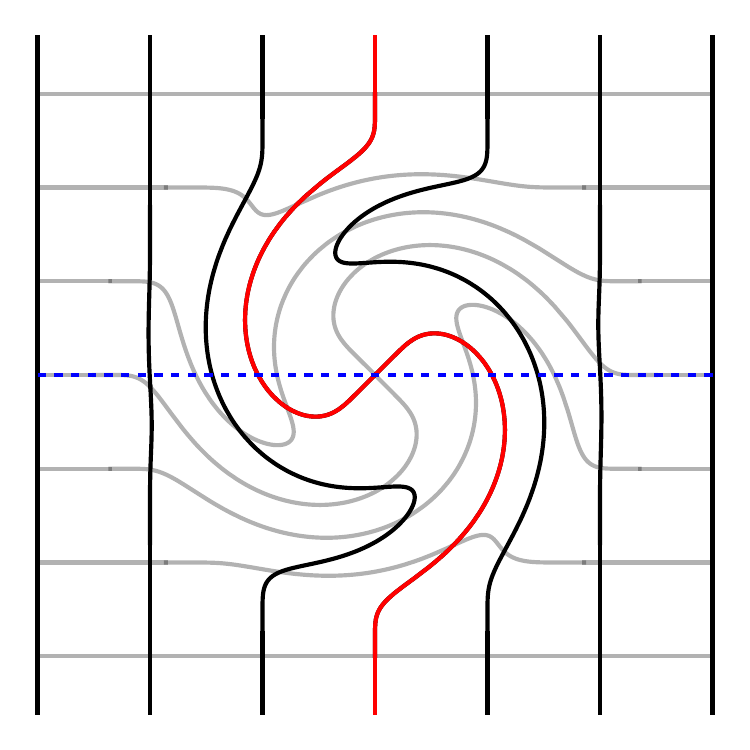}
\caption{Cartesian grid in the complex plane under the diffeomorphism $F(r e^{i \theta}) = r e^{i (\theta + \chi(r) \theta_0)}$
where $\theta_0 = 3 \pi / 4$ and $\chi$ is a smooth, monotone cutoff function from $\mR$ to $[0,1]$ satisfying $\chi(0) = 1$ and $\chi(r)=0$ for $r > 1$. Pushing $g_\Min$ on $\mR^{1+1} = \mC$ forward through $F$, we get a metric $g$ satisfying (1.1) and (1.2). The red curve is the timelike $g$-geodesic $t \mapsto F(t,0)$. It intersects three times the set $\{t = 0\}$ (in blue). 
}
\label{fig_twirl}
\end{figure}

\begin{proof}[Proof of Proposition \ref{prop_tau_proper} {\rm (b)}]
We first claim that for any $r \in \mR$, 
\begin{equation} \label{sigmar_cylinder_compact}
\text{$\Sigma_r \cap (\mR \times \ol{B})$ is compact.}
\end{equation}
To prove this, define 
\[
T = \max_{x \in \p B} |\alpha(r,x)|.
\]
For any $z \in \mR^{1+n}$, let $\eta_z$ be the integral curve of $-\mathrm{grad}_g(\tau)$ with $\eta_z(0) = z$. Note that $\eta_z$ is timelike and future-directed, depends smoothly on $z$, and  $\tau$ is strictly increasing along $\eta_z$. Since $\Sigma_r$ is a Cauchy surface, each $\eta_z$ meets $\Sigma_r$ transversally at a unique time $S(z)$, and $S$ is smooth by Lemma \ref{lemma_intersection}. Consider the smooth map  
\[
F: \{ t = 0 \} \to \Sigma_r, \ \ F(z) = \eta_z(S(z)).
\]
Lemma \ref{lemma_timelike_intersection} ensures that for any $w \in \Sigma_r$, the curve $\eta_w$ meets $\{t=0\}$ at some $z$, and then $w = F(z)$. Thus $F$ is surjective. Now if $z = (0,x)$ with $|x| > T + 1$ and if the curve $\eta_z$ meets $\mR \times \p B$, then the first point of intersection is in the set $\{ t > T \}$ by \eqref{metric_a1}. Moreover, $\tau > r$ in the set $(\mR \times \p B) \cap \{ t > T \}$ by the definition of $T$ and the fact that $\alpha$ is strictly increasing in $r$. Similarly, for negative times $\eta_z$ first meets $\mR \times \p B$ (if it meets $\mR \times \p B$ at all) in a set where $\{ \tau < r \}$. Thus $\eta_{(0,x)}$ for $|x| > T+1$ cannot meet $\Sigma_r$ inside $\mR \times \ol{B}$. It follows that 
\[
\Sigma_r \cap (\mR \times \ol{B}) \subset F(\{ (0,x) \,:\, |x| \leq T+1 \}).
\]
The latter set is compact by continuity of $F$, which proves \eqref{sigmar_cylinder_compact}.

Next we prove that 
\begin{equation} \label{components}
\text{ $\mR^{1+n} \setminus \Sigma_r$ has exactly two components given by $\{ \tau > r \}$ and $\{ \tau < r \}$.}
\end{equation}
To see that $\{ \tau > r \}$ is connected, fix two points $z_0, z_1 \in \{ \tau > r \}$ and let $\gamma: [0,1] \to \mR^{1+n}$ be some continuous curve joining $z_0$ and $z_1$. Let $\rho = \min\{ \tau(z_0), \tau(z_1) \} > r$. The set $\{ t \in [0,1] \,:\, \tau(\gamma(t)) < \rho \}$ is an open subset of $(0,1)$, hence it is the union of disjoint open intervals $(a_j, b_j)$. Since $\Sigma_{\rho}$ is connected, we may replace each $\gamma|_{[a_j,b_j]}$ by a continuous curve in $\Sigma_{\rho}$ with the same endpoints. By this procedure we obtain a continuous curve in $\{ \tau > r \}$ joining $z_0$ and $z_1$, which shows that $\{ \tau > r \}$ is connected. The argument for $\{ \tau < r \}$ is analogous. These are maximal connected subsets of $\mR^{1+n} \setminus \Sigma_r$, since any strictly larger set contains a point both in $\{ \tau > r \}$ and $\{ \tau < r \}$, and such points cannot be connected by a continuous curve in $\mR^{1+n} \setminus \Sigma_r$. This proves \eqref{components}.

Finally we prove Proposition \ref{prop_tau_proper} (b). Suppose that $(t,x) \in \mR \times \ol{B}$ and $\tau(t,x) \leq R$. We will show that there is $T = T(R)$ so that $t \leq T$. This proves the statement in (b) for the $+$ sign. The proof for $-$ sign is analogous, and together these facts imply the properness of $\tau|_{\mR \times \ol{B}}$.

By \eqref{sigmar_cylinder_compact} one has $\Sigma_R \cap (\mR \times \ol{B}) \subset (-T,T) \times \ol{B}$ for some $T > 0$. Then $[T,\infty) \times \ol{B}$ is a connected subset of $\mR^{1+n} \setminus \Sigma_R$, and \eqref{components} implies that it must be contained either in $\{ \tau > R \}$ or in $\{ \tau < R \}$. However, since $\tau$ is strictly increasing along $t \mapsto (t,x)$ for $x \in \p B$, one must have $\tau(t,x) > R$ when $x \in \p B$ and $t$ is sufficiently large. This implies that 
\[
[T,\infty) \times \ol{B} \subset \{ \tau > R \}.
\]
Consequently $t \leq T$ whenever $(t,x) \in \mR \times \ol{B}$ and $\tau(t,x) \leq R$.
\end{proof}


\section{The SPW scattering relation} \label{sec_geometric}

In this section we will prove Theorem \ref{thm_geometric}. It is enough to consider the case where $\omega = e_1$, so that 
\[
\Sigma_{\pm} = \{ (t, x_1,x') \in \mR^{1+n} \,:\, x_1 = \pm 1 \}, \qquad \Sigma_{\pm,\sigma} = \Sigma_{\pm} \cap \{ t = \sigma \}.
\]
Recall that for $z \in \Sigma_-$, $\gamma_z: (-\infty,\rho(z)) \to \mR^{1+n}$ is the geodesic with $\gamma_z(0) = z$ and $\dot{\gamma}_z(0) = (1,e_1)$, $r_+(z)$ is the time when $\gamma_z$ exits $\{ x_1 \leq 1 \}$, $\mathcal{T}_g = \{ z \in \Sigma_- \,:\, r_+(z) = \rho(z) \}$ is the trapped set, and the SPW scattering relation is the map 
\[
\beta_g^{\lambda}: \Sigma_- \setminus \mathcal{T}_g \to T^* \Sigma_+, \ \ z \mapsto (\gamma_z(r_+(z)), \lambda(\gamma_z(r_+(z))) \dot{\gamma}_z(r_+(z))).
\]
We start by proving a weak version of Theorem \ref{thm_geometric} that gives a diffeomorphism property from $\Sigma_-$ to $\Sigma_+$ and includes a nontrapping condition.

\begin{Proposition} \label{prop_nontrapping}
Let $g$ be a Lorentzian metric in $\mR^{1+n}$ satisfying \eqref{metric_a1},  let $\lambda \in C^{\infty}(\Sigma_+)$ be positive, and suppose that for any $\sigma \in \mR$ one has 
\[
\beta_g(\Sigma_{-,\sigma} \setminus \mathcal{T}_g) = \beta_{\gm}^{\lambda}(\Sigma_{-,\sigma}).
\]
Then 
\begin{itemize}
\item 
$\mathcal{T}_g = \emptyset$ and each $\gamma_z$ is defined on all of $\mR$;
\item 
$r_+$ is smooth on $\Sigma_-$;
\item 
 the map $z \mapsto \gamma_z(r_+(z))$ is a diffeomorphism $\Sigma_{-,\sigma} \to \Sigma_{+,\sigma+2}$ for any $\sigma$; and 
 \item
 $\dot{\gamma}_z(r_+(z)) = \lambda(\gamma_z(r_+(z))) (1,e_1)$ for all $z \in \Sigma_-$.
 \end{itemize}
\end{Proposition}

For the proof we formulate two lemmas that will also be needed later when studying the eikonal equation in Section \ref{sec_eikonal}. For completeness we also state the Hadamard global inverse function theorem, which will be used in several places in this article to prove global diffeomorphism properties.

\begin{Theorem}[{\cite[Theorem 6.2.8]{krantz2002}}] \label{thm_hadamard}
Let $M_1$ and $M_2$ be smooth connected manifolds having the same dimension, and let $f: M_1 \to M_2$ be a smooth map. If $f$ is proper, the derivative $D f$ is invertible everywhere on $M_1$, and $M_2$ is simply connected, then $f$ is a diffeomorphism.
\end{Theorem}

The first lemma is a simple result on smoothness of intersection times.

\begin{Lemma} \label{lemma_intersection}
Let $\mathcal{Z}$ be a manifold, and assume that we have a smooth map 
\[
\mathcal{Z} \times \mR \to \mR^{1+n}, \ \ (z,r) \mapsto \eta_z(r).
\]
Let $\Sigma$ be a smooth hypersurface in $\mR^{1+n}$, and suppose that 
\[
\eta_{z_0}(r_0) \in \Sigma, \qquad \dot{\eta}_{z_0}(r_0) \text{ is transverse to $\Sigma$.}
\]
Then there is a smooth map $r(z)$ near $z_0$ such that $r(z_0) = r_0$ and 
\[
\eta_z(r) \in \Sigma \text{ for $(z,r)$ near $(z_0,r_0)$} \quad \Longleftrightarrow \quad r = r(z).
\]
\end{Lemma}
\begin{proof}
Let $\rho: \mR^{1+n} \to \mR$ be a smooth defining function for $\Sigma$, so that $\Sigma = \rho^{-1}(0)$ and $\nabla \rho(z)$ is a nonzero vector normal to $\Sigma$ for any $z \in \Sigma$. We define the smooth function 
\[
F: \mathcal{Z} \times \mR \to \mR^{1+n}, \ \  F(z,r) = \rho(\eta_z(r)).
\]
Then $F(z_0,r_0) = 0$ and $\p_r F(z_0,r_0) = \nabla \rho(\eta_{z_0}(r_0)) \cdot \dot{\eta}_{z_0}(r_0)$. Since $\dot{\eta}_{z_0}(r_0)$ is transverse to $\Sigma$, we have $\p_r F(z_0,r_0) \neq 0$. By the implicit function theorem, there is a smooth function $r(z)$ near $z_0$ such that $r(z_0) = r_0$ and 
\[
F(z,r) = 0 \text{ near $(z_0,r_0)$} \quad \Longleftrightarrow \quad r = r(z).
\]
Since $F(z,r) = 0$ is equivalent to $\eta_z(r) \in \Sigma$, the lemma follows.
\end{proof}

The second lemma gives a nontrapping property for a family of geodesics with certain properties. It will also be needed in Section \ref{sec_eikonal}.

\begin{Lemma} \label{lemma_kappa}
Let $g$ be a Lorentzian metric on $\mR^{1+n}$, and let $\eta_z$ be a family of geodesics depending smoothly on $z \in \Sigma_-$ with the following properties.
\begin{enumerate}
\item[(i)]
Any $\eta_z$ that reaches $\Sigma_+$ meets both $\Sigma_{\pm}$ exactly once, with $\eta_z(0) = z$ and $\eta_z(r_+(z)) \in \Sigma_+$.
\item[(ii)]
Any $w \in \Sigma_+$ lies on a unique $\eta_{z}$ where $z \in \Sigma_-$, which defines a map $\kappa: \Sigma_+ \to \Sigma_-$, $w \mapsto z$.
\item[(iii)]
$\dot{\eta}_{\kappa(w)}(0)$ and $\dot{\eta}_{\kappa(w)}(r_+(\kappa(w)))$ are transverse to $\Sigma_{\pm}$ and depend smoothly on $w$.
\end{enumerate}
Then $\kappa: \Sigma_+ \to \Sigma_-$ is an embedding. If $\kappa$ is also surjective or proper, then it is a diffeomorphism, 
and for any $z \in \Sigma_-$ the geodesic $\eta_z$ meets $\Sigma_+$ at a unique time $r_+(z)$ depending smoothly on $z$.
\end{Lemma}
\begin{proof}
If $w \in \Sigma_+$, write $r(w) = r_+(\kappa(w))$ and $v_{+}(w) = \dot{\eta}_{\kappa(w)}(r(w))$ and let $\gamma_{z,v}$ be the $g$-geodesic through $z$ in direction $v$. Uniqueness of geodesics and (i) give that $\gamma_{w,v_+(w)}(s) = \eta_{\kappa(w)}(r(w) + s)$. Thus in particular 
\[
\gamma_{w,v_+(w)}(-r(w)) \in \Sigma_-, \qquad \dot{\gamma}_{w,v_+(w)}(-r(w)))\text{ is transverse to $\Sigma_-$}.
\]
By (iii) $\gamma_{w,v_+(w)}$ depends smoothly on $w$ and by (i) it meets $\Sigma_-$ only at time $-r(w)$. Lemma \ref{lemma_intersection} implies that $r(w)$ depends smoothly on $w$. Thus the map 
\[
\kappa: \Sigma_+ \to \Sigma_-, \ \ \kappa(w) = \gamma_{w,v_+(w)}(-r(w)),
\]
is smooth on $\Sigma_+$.

On the other hand, fix $\bar{w} \in \Sigma_+$ and write $\bar{z} = \kappa(\bar{w}) \in \Sigma_-$. Since 
\[
\eta_{\bar{z}}(r_+(\bar{z})) \in \Sigma_+, \qquad \dot{\eta}_{\bar{z}}(r_+(\bar{z}))  \text{ is transverse to $\Sigma_+$},
\]
Lemma \ref{lemma_intersection} ensures that there is a smooth function $R(z)$ for $z$ near $\bar{z}$ such that 
\[
\eta_z(R(z)) \in \Sigma_+ \text{ for $z$ near $\bar{z}$.}
\]
By (i) one has $R(z) = r_+(z)$ and $\kappa(\eta_z(R(z))) = z$ near $\bar{z}$. This implies that $D\kappa$ is invertible near $\bar{w}$, and since $\bar{w} \in \Sigma_+$ was arbitrary we see that $\kappa$ is a local diffeomorphism. If $\kappa(w_1) = \kappa(w_2)$ then (i) implies that $w_1 = w_2$, showing that $\kappa$ is injective and hence an embedding.

If $\kappa$ is also surjective, then it is a diffeomorphism $\Sigma_+ \to \Sigma_-$. On the other hand, if $\kappa$ is proper, the Hadamard global inverse function theorem (Theorem \ref{thm_hadamard}) implies that $\kappa$ is a diffeomorphism. Then any $z \in \Sigma_-$ is of the form $z = \kappa(w)$ for a unique $w \in \Sigma_+$. By (ii) $\eta_z$ reaches $\Sigma_+$, and by (i) $\eta_z$ meets $\Sigma_+$ at the unique time $r_+(z)$. We proved above that $r_+(z) = R(z)$ is smooth.
\end{proof}

\begin{proof}[Proof of Proposition \ref{prop_nontrapping}]
For the metric $\gm$, the geodesics starting at $\Sigma_-$ in direction $(1,e_1)$ are the straight lines $r \mapsto z + r(1, e_1)$. This shows that  
\[
\beta_{\gm}^{\lambda}(\Sigma_{-,\sigma}) = \{ (w, \lambda(w)(1,e_1)) \,:\, w \in \Sigma_{+,\sigma+2} \}.
\]
The assumption $\beta_{g}(\Sigma_{-,\sigma} \setminus \mathcal{T}_g) = \beta_{\gm}^{\lambda}(\Sigma_{-,\sigma})$ for all $\sigma$ implies that 
\[
\{ (\gamma_z(r_+(z)), \dot{\gamma}_z(r_+(z))) \,:\, z \in \Sigma_{-,\sigma} \setminus \mathcal{T}_g \} = \{ (w, \lambda(w)(1,e_1)) \,:\, w \in \Sigma_{+,\sigma+2} \}.
\]
We can read off two facts:
\begin{itemize}
\item 
For any $z \in \Sigma_{-,\sigma}$ with $z \notin \mathcal{T}_g$, one has $\gamma_z(r_+(z)) \in \Sigma_{+,\sigma+2}$ and $\dot{\gamma}_z(r_+(z)) = \lambda(1,e_1)$.
\item 
Any $w \in \Sigma_{+,\sigma+2}$ lies on some geodesic $\gamma_z$ with $z \in \Sigma_{-,\sigma}$.
\end{itemize}
Combining this with \eqref{metric_a1}, which implies that geodesics are straight lines in  $\{ |x_1| \geq 1 \}$ and that any $\gamma_z$ that reaches $\Sigma_+$ from $\{ x_1 < 1 \}$ must stay in $\{ x_1 > 1 \}$ afterwards, we obtain that any geodesic $\gamma_z$ with $z \notin \mathcal{T}_g$ intersects both $\Sigma_-$ and $\Sigma_+$ exactly once. Moreover, the tangent vector on $\Sigma_-$ is $(1,e_1)$ and on $\Sigma_+$ it is $\lambda(w)(1,e_1)$. By uniqueness of geodesics, it follows that any $w \in \Sigma_{+,\sigma+2}$ lies on exactly one geodesic $\gamma_{\kappa(w)}$ where $\kappa(w) \in \Sigma_-$ (then $\kappa(w) \in \Sigma_{-,\sigma}$).

We have proved that the family $\gamma_z$ satisfies the conditions in Lemma \ref{lemma_kappa}. Thus $\kappa: \Sigma_+ \to \Sigma_-$ is an embedding. Then also $\kappa: \Sigma_{+,\sigma+2} \to \Sigma_{-,\sigma}$ is an embedding, and in particular $\kappa(\Sigma_{+,\sigma+2})$ is open in $\Sigma_{-,\sigma}$. We write 
\[
\kappa(\Sigma_{+,\sigma+2}) = \kappa(\Sigma_{+,\sigma+2} \cap \{ |x'| \geq 1 \}) \cup \kappa(\Sigma_{+,\sigma+2} \cap \{ |x'| \leq 1 \}).
\]
By \eqref{metric_a1}, for $|x'| \geq 1$ one has $\kappa(\sigma+2,1,x') = (\sigma,-1,x')$. Thus $\kappa(\Sigma_{+,\sigma+2} \cap \{ |x'| \geq 1 \})$ is closed. On the other hand, since $\kappa$ is continuous, $\kappa(\Sigma_{+,\sigma+2} \cap \{ |x'| \leq 1 \})$ is compact and hence closed. Thus $\kappa(\Sigma_{+,\sigma+2})$ is an open and closed set, so by connectedness $\kappa(\Sigma_{+,\sigma+2}) = \Sigma_{-,\sigma}$ and $\kappa$ is surjective. By Lemma \ref{lemma_kappa} one has $\mathcal{T}_g = \emptyset$, $r_+$ is smooth, and the map $\kappa^{-1}: z \mapsto \gamma_z(r_+(z))$ is a diffeomorphism mapping $\Sigma_{-,\sigma}$ onto $\Sigma_{+,\sigma+2}$. Since each $\gamma_z$ reaches $\Sigma_+$ with tangent vector $\dot{\gamma}_z(r_+(z)) = \lambda (1,e_1)$, \eqref{metric_a1} implies that each $\gamma_z$ is defined in all of $\mR$.
\end{proof}

The next proposition shows that the set $J^-(\Sigma_{+,\sigma}) \cap \Sigma_-$ agrees with the corresponding set for $\gm$, if the SPW scattering relations agree. Contrary to Proposition \ref{prop_nontrapping}, both \eqref{metric_a1} and \eqref{metric_a2} are assumed, not only \eqref{metric_a1}.

\begin{Proposition} \label{prop_jminus}
Let $g$ be satisfy \eqref{metric_a1}--\eqref{metric_a2}, let $\lambda \in C^{\infty}(\Sigma_+)$ be positive, and suppose that for any $\sigma \in \mR$ one has 
\[
\beta_g(\Sigma_{-,\sigma} \setminus \mathcal{T}_g) = \beta_{\gm}^{\lambda}(\Sigma_{-,\sigma}).
\]
Then 
\[
J^-(\Sigma_{+,\sigma}) \cap \Sigma_- = \{t \le \sigma-2\} \cap \Sigma_-.
\]
\end{Proposition}

Again the proof will be based on several lemmas. Below we will write 
\[
\mathcal{C} = \mR \times B.
\]
The first lemma is related to points on $\p \mathcal{C}$.

\begin{Lemma} \label{lemma_c_bd}
Let $g$ satisfy \eqref{metric_a1} and let $p_j = (t_j, x_j) \in \p \mathcal{C}$. If $t_1 \leq t_2-\pi$, then $p_1 \leq p_2$. If $g$ also satisfies \eqref{metric_a2}, then $p_1 \leq p_2$ implies $t_1 < t_2 + \pi$.
\end{Lemma}
\begin{proof}
Let $t_1 \leq t_2-\pi$ and let $\eta: [0,L] \to S^{n-1}$ be a unit speed length minimizing geodesic on $S^{n-1}$ from $x_1$ to $x_2$. Then $L \leq \pi$, and by \eqref{metric_a1} the curve $\gamma(r) = (t_1+r,\eta(r))$ is a future-directed causal curve from $p_1$ to $(t_1+L,x_2)$. Further, joining $(t_1+L,x_2)$ to $(t_2,x_2)$ by the line segment $r \mapsto (t_1+L+r,x_2)$ shows that $p_1 \leq p_2$.

Now assume that also \eqref{metric_a2} holds. We argue by contradiction and suppose that $p_1 \leq p_2$ but $t_1 \geq t_2 + \pi$. Then $t_2 \leq t_1 - \pi$, and the first part implies that $p_2 \leq p_1$. The Cauchy temporal function $\tau$ is strictly increasing on future-directed causal curves. Thus if one had $p_1 \neq p_2$, then the conditions $p_1 \leq p_2$ and $p_2 \leq p_1$ would yield $\tau(p_1) < \tau(p_2)$ and $\tau(p_2) < \tau(p_1)$, which is impossible. Therefore $p_1 = p_2$, which contradicts the assumption $t_1 \geq t_2 + \pi$.
\end{proof}

The next lemma gives a decomposition of $J^-(\Sigma_{+,\sigma})$. Below, $J_{\mathrm{Min}}^{-}$ will denote the causal past with respect to $\gm$.

\begin{Lemma}\label{near_far_J}
Let $g$ satisfy \eqref{metric_a1}--\eqref{metric_a2}, let $\sigma \in \R$ and $R \geq 2 \pi + 1$, and define
\[
A = \Sigma_{+,\sigma} \cap \{|x| \ge R\},
\quad
K = \Sigma_{+,\sigma} \cap \{|x| \le R\}.
\]
Then
\[
J^-(\Sigma_{+,\sigma}) = (J_{\mathrm{Min}}^-(A) \setminus \mathcal C) \cup J^-(K).
\]
\end{Lemma}
\begin{proof}
Let $p \in J^-(\Sigma_{+,\sigma})$
and let us show that $p \in (J_\Min^-(A) \setminus \mathcal C) \cup J^-(K)$.
There is $q \in \Sigma_{+,\sigma}$ and a past-directed causal curve $\gamma$ from $q$ to $p$. If $q \in K$, then $p \in J^-(K)$. Now suppose that $q \in A$.
If $\gamma$ does not intersect $\mathcal C$, 
then $\gamma$ is causal with respect to $\gm$ and 
$p \in J_\Min^-(A) \setminus \mathcal C$.
Suppose now that $\gamma$ intersects $\mathcal C$ and let $\tilde p = (t_0, x_0) \in \p \mathcal C$ be the first intersection point with $\p \mathcal C$. 
There holds $t_0 \leq \sigma - (R-1) \leq \sigma - 2 \pi$, as $\gamma$ is causal with respect to $\gm$ outside $\mathcal C$ and $R \geq 2 \pi + 1$. By Lemma \ref{lemma_c_bd} one has $\tilde{p} \leq \tilde{q}$ for $\tilde{q} = (\sigma, e_1)$. Thus $p \le \tilde p \le \tilde q$. Since $\tilde{q} \in K$, we have $p \in J^-(K)$.

Let now $p \in (J_\Min^-(A) \setminus \mathcal C) \cup J^-(K)$
and let us show that $p \in J^-(\Sigma_{+,\sigma})$. Observe that 
\[
J^-(\Sigma_{+,\sigma}) = J^-(A \cup K) = J^-(A) \cup J^-(K).
\]
If $p \in J^-(K)$, then $p \in J^-(\Sigma_{+,\sigma})$. On the other hand, suppose that 
$p \in J_\Min^-(A) \setminus \mathcal C$.
There is $q \in A$ and a curve $\gamma$ from $q$ to $p$ that is causal with respect to $\gm$. 
If $\gamma$ does not intersect $\mathcal C$, then $\gamma$ is causal with respect to $g$ and $p \in J^-(A)$.
Suppose now that $\gamma$ intersects $\mathcal C$, let $p_1 = (t_1,x_1)$ be the first intersection point with $\p \mathcal C$,  and let $p_2 = (t_2, x_2) \in \p \mathcal C$ be the last intersection point with $\p \mathcal C$. One has $t_1 \leq \sigma - (R-1) \leq \sigma - 2\pi$. Moreover, since $p_2 \leq p_1$, Lemma \ref{lemma_c_bd} implies that $t_2 < t_1 + \pi \leq \sigma - \pi$. Using Lemma \ref{lemma_c_bd} again we see that $p_2 \leq q$ where $q = (\sigma, e_1)$. Thus $p_2 \in J^-(K)$. Moreover, $p \le p_2$ since the segment of $\gamma$ from $p_2$ to $p$ is outside $\mathcal C$ due to $p$ being outside $\mathcal C$.
\end{proof}

\begin{Lemma}\label{lem_cl_bdd}
Let $g$ satisfy \eqref{metric_a1}--\eqref{metric_a2} and let $\sigma \in \R$.
The set $J^-(\Sigma_{+,\sigma})$ is closed, and 
for any $z \in \Sigma_-$ there is $r > 0$ such that 
$z + r \bar{e}_0 \not\in J^-(\Sigma_{+,\sigma})$ where $\bar{e}_0 = (1,0) \in \mR^{1+n}$. 
\end{Lemma}
\begin{proof}
Let $R \geq 2 \pi + 1$ and let $A$ and $K$ be as in Lemma \ref{near_far_J}.
Due to the properties of the Minkowski geometry, $J_\Min^-(A)$ is closed. 
The set $J^-(K)$ is closed as $K$ is compact, see \cite[Theorem 2.1]{hounnonkpe2019}.  
As $\mathcal C$ is open, the closedness of $J^-(\Sigma_{+,\sigma})$ follows from Lemma \ref{near_far_J}. 

If $z \in \Sigma_-$, the future $J^+(z)$ is closed by \cite[Theorem 2.1]{hounnonkpe2019}, so $J^+(z) \cap J^-(K)$ is also closed. Moreover, one has 
\[
J^+(z) \cap J^-(K) \subset J^+(K \cup \{z\}) \cap J^-(K \cup \{z\})
\]
and the latter set is compact \cite[Proposition 2.3]{hounnonkpe2019}. Thus $J^+(z) \cap J^-(K)$ is compact, and there is $T > 0$ such that
    \begin{align}
 J^+(z) \cap J^-(K) \subset \{t \le T\}.
    \end{align}
We see that $z + r \bar{e}_0 \not\in J_\Min^-(A) \cup J^-(K)$ for $r \gg 0$.
\end{proof}

\begin{Lemma}\label{lem_bd_null}
The boundary of $J^-(\Sigma_{+,\sigma})$, $\sigma \in \R$, is covered by normal null geodesics from $\Sigma_{+,\sigma}$.
\end{Lemma}
\begin{proof}
By Lemma \ref{lem_cl_bdd} and \cite[Lemma 6, p.\ 404]{oneill1983} applied to $J^-(\Sigma_{+,\sigma})$, we have 
    \begin{align}
\p J^-(\Sigma_{+,\sigma}) = J^-(\Sigma_{+,\sigma}) \setminus I^-(\Sigma_{+,\sigma}),
    \end{align}
where $I^-(\Sigma_{+,\sigma})$ is the chronological past of $\Sigma_{+,\sigma}$, that is, the set of points $p \in M$ such that there is a future-directed timelike curve from $p$ to a point in $\Sigma_{+,\sigma}$.
Let $p \in \p J^-(\Sigma_{+,\sigma})$. Then there is a causal curve from $p$ to $\Sigma_{+,\sigma}$. This curve is a normal null geodesic, since otherwise we get the contradiction $p \in I^-(\Sigma_{+,\sigma})$ from  \cite[Theorem 51, p.\ 298]{oneill1983}. 
\end{proof}

\begin{proof}[Proof of Proposition \ref{prop_jminus}]
Observe that the past-directed normal null geodesics from $\Sigma_{+,\sigma}$ are the geodesics starting in direction $c(-1,\pm e_1) \in \R^{1+n}$ for $c > 0$. By \eqref{metric_a1} the geodesics in direction $(-1,e_1)$ never reach $\Sigma_-$. Thus Lemma \ref{lem_bd_null} ensures that $\p J^-(\Sigma_{+,\sigma}) \cap \Sigma_-$ is covered by geodesics starting from $\Sigma_{+,\sigma}$ in direction $-c(1,e_1)$. In view of Proposition \ref{prop_nontrapping} such geodesics meet $\Sigma_-$ only in the set $\Sigma_{-,\sigma-2}$. Thus we have 
    \begin{align}\label{upper}
\p J^-(\Sigma_{+,\sigma}) \cap \Sigma_- \subset \Sigma_{-,\sigma-2}.
    \end{align}
Let $z \in \Sigma_{-,\sigma-2}$. By Proposition \ref{prop_nontrapping} we have $z \in J^-(\Sigma_{+,\sigma})$.
Consider the curve $\mu(r) = z + r (1,0)$.
If $r \le 0$, then $\mu(r) \le z$ and therefore $\mu(r) \in J^-(\Sigma_{+,\sigma})$.
On the other hand, let $r > 0$. 
By Lemma \ref{lem_cl_bdd}, there is $R > 0$ such that $\mu(R) \not\in J^-(\Sigma_{+,\sigma})$. In view of \eqref{upper}, the curve $\mu$ can intersect $\p J^-(\Sigma_{+,\sigma})$ at most at $\mu(0)$. Hence the set $\{ r > 0 \,:\, \mu(r) \notin \ol{J^-(\Sigma_{+,\sigma})} \}$ is a nonempty open and closed subset of $\{ r > 0 \}$. By connectedness $\mu(r) \not\in J^-(\Sigma_{+,\sigma})$ when $r > 0$.
\end{proof}

The final proposition needed for the proof of Theorem \ref{thm_geometric} states that $\Phi$ is a local diffeomorphism.

\begin{Proposition} \label{prop_phi_local_diffeo}
Let $g$ satisfy \eqref{metric_a1}--\eqref{metric_a2}, let $\lambda \in C^{\infty}(\Sigma_+)$ be positive, and suppose that for any $\sigma \in \mR$ one has 
\[
\beta_g(\Sigma_{-,\sigma} \setminus \mathcal{T}_g) = \beta_{\gm}^{\lambda}(\Sigma_{-,\sigma}).
\]
Then the derivative of $\Phi: \Sigma_- \times \mR \to \mR^{1+n}$, $\Phi(z,r) = \gamma_z(r)$ is invertible everywhere.
\end{Proposition}
\begin{proof}
We write $z = (t,-1,y) \in \Sigma_-$ with $y \in \mR^{n-1}$ and $\Phi(z,r) = \Phi(t,y,r)$. To get a contradiction, suppose that $\Phi$ is singular at $(z_*, r_*)  \in \Sigma_- \times \mR$ 
with $z_* = (t_*, -1, y_*)$, that is, 
there is a direction $(\delta t, \delta y, \delta r) \in \R^{1 + n}$ such that at $(z_*, r_*)$ there holds
    \begin{align}\label{singular_pt_phi}
\p_t \Phi \delta t + \p_{y_j} \Phi \delta y_j + \p_r \Phi \delta r = 0. 
    \end{align}

Define the vector field $J(r) = \p_t \Phi(z_*, r)$ along $\gamma(r) = \gamma_{z_*}(r)$. Since $D_r \p_t = D_t \p_r$ where $D$ denotes the covariant derivative and since $\gamma$ is a null geodesic, we have  
\[
\p_r (g(J, \dot{\gamma})) = g(D_r \p_t \Phi, \dot{\gamma}) + g(J, D_r \dot{\gamma}) = g(D_t \p_r \Phi, \dot{\gamma}) = \frac{1}{2} \p_t(g(\dot{\gamma}, \dot{\gamma})) = 0.
\]
Thus $g(J, \dot \gamma) = g(J, \dot \gamma)|_{r=0} = g(\p_t, N) = -1$ where $N = (1,e_1)$. A similar argument shows that $g(\p_{y_j} \Phi, \dot{\gamma}) = 0$.

Returning to \eqref{singular_pt_phi}, and using the fact that $\dot \gamma = \p_r \Phi$ is a null vector, we have
\[
0 = g(\p_t \Phi \delta t + \p_{y_j} \Phi \delta y_j + \p_r \Phi \delta r, \dot \gamma) = -\delta t. 
\]
Thus $\delta t = 0$, and the map 
\[
\Phi_* :  \mR^{n-1} \times \mR \to \mR^{1+n}, \quad \Phi_*(y,r) = \Phi((t_*, -1, y), r)
\]
is singular at $(y_*, r_*)$ in the sense that $D \Phi_*$ does not have full rank. 

Let us show that $r_* < r_+(z_*)$.
For $z = (t_*, -1, y)$, by Proposition \ref{prop_nontrapping} we can write $\gamma_z(r_+(z)) = (t_* + 2, 1, \psi_*(y))$ where $\psi_*$ is a diffeomorphism of $\mR^{n-1}$. Observe that if $r \ge r_+(t_*, -1, y)$, then the facts that $\dot{\gamma}_z(r_+(z)) = \lambda(1,1,0)$ and \eqref{metric_a1} imply that 
    \begin{align}
\Phi_*(y,r) = (t_* + 2, 1, \psi_*(y)) + \lambda(t_*+2, 1, \psi_*(y)) (r - r_+(t_*, -1, y)) (1, 1, 0).
    \end{align}
Thus for $r \ge r_+(t_*, -1, y)$ the derivative $D \Phi_*$ has the form 
    \begin{align}
\begin{pmatrix}
* & \lambda
\\
* & \lambda
\\
D_y \psi_* & 0
\end{pmatrix}.
    \end{align}
Regardless of the elements marked with $*$ that were not computed explicitly, the columns of $D \Phi_*$ are seen to be linearly independent since the columns of $D_y \psi_*$ are. Thus $r_* < r_+(z_*)$.

Observe that the null geodesic $\gamma$ is normal to the spacelike submanifold $\Sigma_{-,t_*}$. 
Moreover, $\gamma(r_*)$ is a focal point of $\Sigma_{-,t_*}$ along $\gamma$ since $\Phi_*$ is singular at $(y_*, r_*)$, see \cite[Proposition 30, p.\ 283]{oneill1983} (the normal exponential map to $\Sigma_{-,t_*}$ satisfies $D \exp(\delta y, N \delta r) = D \Phi_*(\delta y, \delta r)$). By \cite[Proposition 48, p.\ 296]{oneill1983} there is a timelike curve from some point $z_- \in \Sigma_{-,t_*}$ to the point $\gamma(r_+(z_*)) \in \Sigma_{+,t_* + 2}$.
But then $z_-$ is in the interior of $J^-(\Sigma_{+,t_* + 2}) \cap \Sigma_-$, see \cite[Lemma 6, p.\ 404]{oneill1983}.
This is a contradiction since $z_-$ is on the boundary of $J^-(\Sigma_{+,t_* + 2}) \cap \Sigma_-$ according to Proposition \ref{prop_jminus}.
We have shown that $D\Phi$ is invertible everywhere.
\end{proof}

\begin{proof}[Proof of Theorem \ref{thm_geometric}]
Proposition \ref{prop_nontrapping} ensures that $\Phi$ is defined on $\Sigma_- \times \mR$. By the Hadamard global inverse function theorem (Theorem \ref{thm_hadamard}) and Proposition \ref{prop_phi_local_diffeo}, to see that $\Phi: \Sigma_- \times \mR \to \mR^{1+n}$ is a diffeomorphism it is enough to show that it is proper. Let $K \subset \mR^{1+n}$ be compact, so there is $R > 1$ such that 
\begin{equation} \label{geom_k_bound}
K \subset \{ (t,x_1,y) \,:\, |t| \leq R, \, |x_1| \leq R, \, |y| \leq R \}.
\end{equation}
Since the Cauchy temporal function $\tau$ is smooth, there is $T > 0$ such that 
\begin{equation} \label{geom_tau_bound}
|\tau| \leq T \text{ in $K$.}
\end{equation}
Moreover, since the function $\alpha$ in Lemma \ref{lemma_taur_ext} is smooth in $\{ |x| \geq 1 \}$, there is $A > 0$ such that 
\begin{equation} \label{geom_alpha_bound}
\sup_{ |y| \leq R } \alpha(T, (-1,y)) \leq A, \qquad \inf_{|y| \leq R} \alpha(-T, (1,y)) \geq -A.
\end{equation}

To show that $\Phi$ is proper, it is enough to show that there is $C = C(K)> 0$ such that for any $(z,r)$ with $\Phi(z,r) = \gamma_z(r) \in K$ where $z = (t, -1, y)$, one has 
\[
|t| \leq C, \qquad |y| \leq C, \qquad |r| \leq C.
\]
We fix $(z,r)$ with $\gamma_z(r) \in K$. First note that by \eqref{metric_a1} any geodesic $\gamma_{\tilde{z}}$ which begins in $\{ |\tilde{y}| > R \}$ stays forever in $\{ |\tilde{y}| > R \}$. Thus we have  
\[
|y| \leq R.
\]

Next we give a bound for $t$. If $r \leq 0$, then \eqref{metric_a1} gives the identity $\gamma_{z}(r) = z + r(1,e_1)$. Together with \eqref{geom_k_bound} this implies that $|t + r| \leq R$ and $|-1+r| \leq R$. Thus 
\begin{equation} \label{geom_rt_bound}
|r| \leq R+1 \quad \text{ and } \quad |t| \leq 2R + 1.
\end{equation}
Similarly, if $r \geq r_+(z)$, then $\gamma_{z}(r) = \gamma_{z}(r_+(z)) + \lambda(\gamma_z(r_+(z))) (r-r_+(z))(1,e_1)$ by \eqref{metric_a1} and the fact that $\dot{\gamma}_z(r_+(z)) = \lambda(1,e_1)$. Proposition \ref{prop_nontrapping} ensures that $\gamma_{z}(r_+(z)) \in \Sigma_{+,t+2}$. Using \eqref{geom_k_bound} we get $|t+2 + \lambda(r-r_+(z))| \leq R$ and $|1+\lambda(r-r_+(z))| \leq R$. This implies 
\begin{equation} \label{geom_rmt_bound}
|\lambda(r-r_+(z))| \leq R+1  \quad \text{ and } \quad  |t| \leq 2R + 3.
\end{equation}
Finally, if $0 < r < r_+(z)$, the fact that $\tau$ is increasing along $\gamma_{z}$ implies that 
\[
\tau(z) \leq \tau(\gamma_{z}(r)) \leq \tau(\gamma_{z}(r_+(z))).
\]
By \eqref{geom_tau_bound} we have $|\tau(\gamma_{z}(r))| \leq T$. In particular $\tau(z) \leq T$. By \eqref{geom_alpha_bound} we have 
\[
t = \alpha(\tau(z), (-1,y)) \leq \alpha(T, (-1,y)) \leq A.
\]
On the other hand, by \eqref{geom_tau_bound} we have $-T \leq \tau(\gamma_{z}(r)) \leq \tau(\gamma_{z}(r_+(z)))$. Since $\gamma_{z}(r_+(z))$ is of the form $(t+2,1,\tilde{y})$ where $|\tilde{y}| \leq C$, \eqref{geom_alpha_bound} gives  
\[
t+2 = \alpha(\tau(\gamma_{z}(r_+(z))), (1,\tilde{y})) \geq \alpha(-T, (1,\tilde{y})) \geq -A.
\]
All in all we have proved that for some $C = C(K)$ one has 
\[
|t| \leq C.
\]

It remains to give a bound for $r$. Since $|t| \leq C$ and $|y| \leq R$, continuity and compactness give that $r_+(z) \leq C$ for some $C= C(K)$ that may change from line to line. Since $\lambda$ is smooth and positive we also have 
\[
C^{-1} \leq \lambda(\gamma_z(r_+(z)) \leq C.
\]
Thus if $\Phi(z,r) \in K \cap \{ |x_1| \leq 1 \}$, then $0 \leq r \leq C$.  On the other hand, if $\Phi(z,r) \in K \cap \{ x_1 <  -1 \}$, then $r < 0$ and \eqref{geom_rt_bound} gives $|r| \leq R+1$. Similarly, if $\Phi(z,r) \in K \cap \{ x_1 > 1 \}$, then $r > r_+(z)$ and \eqref{geom_rmt_bound} gives $|\lambda r| \leq R+1+\lambda r_+ \leq C$, so $|r| \leq C$. This concludes the proof that $\Phi$ is proper and hence a diffeomorphism.

We already know from Proposition \ref{prop_nontrapping} that $\dot{\gamma}_z(r) = \lambda(\gamma_z(r)) (1,e_1)$ when $\gamma_z(r) \in \Sigma_{+}$, and $\dot{\gamma}_z(r) = (1,e_1)$ when $\gamma_z(r) \in \Sigma_{-}$. By \eqref{metric_a1} we have $\dot{\gamma}_z(r) = \hat{\lambda}(1,e_1)$ for some smooth positive $\hat{\lambda}$ whenever $\gamma_z(r) \notin \mR \times \ol{B}$. Moreover, $\hat{\lambda} = 1$ in $\{ x_1 \leq -1 \} \cup \{ |y| \geq 1 \}$, and $\hat{\lambda}(w + s(1,e_1)) = \hat{\lambda}(w)$ for $w \in \Sigma_+$ and $s \geq 0$.
\end{proof}


\section{The eikonal equation} \label{sec_eikonal}

In Theorem \ref{thm_geometric} the conclusion is a diffeomorphism property, which we now characterize in terms of solutions of eikonal equations. Let $\omega \in S^{n-1}$ and recall that $\Sigma_{-} = \{ (t,x) \in \mR^{1+n} \,:\, x \cdot \omega = -1 \}$. Given $z \in \Sigma_-$, let $\gamma_{z}(r)$ be the $g$-geodesic with $\gamma_{z}(0) = z$ and $\dot{\gamma}_{z}(0) = (1,\omega)$ defined in the maximal interval $(-\infty, \rho(z))$. Consider the map 
\[
\Phi: D \to \mR^{1+n}, \ \ \Phi(z, r) = \gamma_{z}(r),
\]
where $D = \{ (z,r) \,:\, z \in \Sigma_-, \, r < \rho(z) \}$ is the maximal domain of $\Phi$.

\begin{Proposition} \label{prop_eikonal}
Let $g$ be a smooth Lorentzian metric in $\mR^{1+n}$ satisfying \eqref{metric_a1}--\eqref{metric_a2}. Then the following conditions are equivalent.
\begin{enumerate}
\item[(a)] 
$\Phi: D \to \mR^{1+n}$ is a diffeomorphism.
\item[(b)]  
There is a smooth function $\varphi = \varphi_{\omega}$ in $\mR^{1+n}$  such that 
\[
g(d\varphi, d\varphi) = 0, \quad \varphi|_{\{x \cdot \omega \leq -1\}} = t - x \cdot \omega.
\]
\end{enumerate}
Moreover, if \emph{(a)} holds, then $D = \Sigma_- \times \mR$, $\varphi = t - x \cdot \omega$ outside $\mR \times B$, and 
\[
d\varphi(\gamma_z(r)) = -\dot{\gamma}_z(r)^{\flat}.
\]
In particular, $d\varphi$ is nowhere vanishing and $\varphi$ is constant along each $\gamma_z$. Moreover, if \emph{(b)} holds both for $\omega$ and $\tilde{\omega}$ and $d\varphi_{\omega}(z_0) = \lambda d\varphi_{\tilde{\omega}}(z_0)$ for some $z_0 \in \mR^{1+n}$ and some $\lambda > 0$, then $\omega = \tilde{\omega}$.
\end{Proposition}

In fact, for our main results we only need the implication (a) $\implies$ (b). This is a standard consequence of the geometric approach to solving eikonal equations via Lagrangian manifolds. For completeness we will give a detailed proof including the implication (b) $\implies$ (a) in Appendix \ref{sec_app}.

We will also need the standard fact that for null bicharacteristics $(z(r), \zeta(r))$ of $\Box_g$, the curves $z(r)$ are null geodesics, and that $\zeta(r) = d\varphi(z(r))$ if $\varphi$ is a suitable smooth solution of the eikonal equation. The proof is given in Appendix \ref{sec_app}. 

\begin{Lemma} \label{lemma_bichar_geodesic}
Let $g$ be a Lorentzian metric in $\mR^{1+n}$ and let $p(z,\zeta) = -\frac{1}{2} g_z(\zeta, \zeta) = -\frac{1}{2} g^{jk}(z) \zeta_j \zeta_k$. Given $(\bar{z}, \bar{\zeta}) \in T^* \mR^{1+n} \setminus 0$, the integral curve of the Hamilton vector field $H_p$ through $(\bar{z}, \bar{\zeta})$ is the curve $(z(r), \zeta(r))$ where $z(r)$ is the geodesic with $z(0) = \bar{z}$ and $\dot{z}(0) = -\bar{\zeta}^{\sharp}$, and $\zeta(r) = -\dot{z}(r)^{\flat}$.

Moreover, if $\varphi$ is a smooth solution of the eikonal equation 
\[
p(z, d\varphi(z)) = 0 \text{ near $z([a,b])$}
\]
and if $\zeta(r_0) = d\varphi(z(r_0))$ for some $r_0 \in [a,b]$, then $\zeta(r) = d\varphi(z(r))$ and $\varphi(z(r))$ is constant for all $r \in [a,b]$.
\end{Lemma}

We give some further properties of $\varphi$.

\begin{Lemma} \label{lemma_eikonal}
Let $g$ satisfy \eqref{metric_a1}, and let $\varphi = \varphi_{\omega}$ be as in Proposition \ref{prop_eikonal}. Then the following properties hold in $\mR^{1+n}$.
\begin{enumerate}
\item[(i)] 
The vector field $Z = -(d\varphi)^{\sharp}$ defined by 
\[ Z u = -g(d\varphi, du) \]
is future-directed, tangential to the surfaces $\{\varphi = s \}$, and satisfies $Z|_{\{x \cdot \omega \leq -1\}} = \p_t + \omega \cdot \nabla_x$. The integral curves of $Z$ are the future-directed null geodesics $\gamma_z(r)$ for $z \in \Sigma_-$. \item[(ii)]
If $\gamma(r)$ is a future-directed causal curve, then $\p_r(\varphi(\gamma(r))) \geq 0$ with equality if and only if $\dot{\gamma}(r) = \lambda Z(\gamma(r))$ for some $\lambda > 0$. Analogously, $\p_r(\varphi(\gamma(r))) \leq 0$ along past-directed causal curves with equality if and only if $\dot{\gamma}(r) = -\lambda Z(\gamma(r))$ for some $\lambda > 0$. 
\end{enumerate}
\end{Lemma}
\begin{proof}
(i) Since $\varphi = t-x \cdot \omega$ in $\{ x \cdot \omega \leq -1 \}$ and $g = \gm$ there, one has $Z|_{\{x \cdot \omega \leq -1\}} = -(d\varphi)^{\sharp}|_{\{x \cdot \omega \leq -1\}} = \p_t + \omega \cdot \nabla_x$. Thus $Z$ is future-directed when $x \cdot \omega \leq -1$. Now if $Z$ would stop being future-directed somewhere along a curve $\gamma_{z}(r)$, then the null vector $Z$ and hence $d\varphi$ would vanish somewhere, which is impossible by Proposition \ref{prop_eikonal}. One has $d\varphi(Z) = Z\varphi = -g(d\varphi,d\varphi) = 0$, so $Z$ is tangential to $\{ \varphi = s \}$. Finally, Proposition \ref{prop_eikonal} gives 
\[
\dot{\gamma}_z(r) = -d\varphi(\gamma_z(r))^{\sharp} = Z(\gamma_z(r)),
\]
which shows that the integral curves of $Z$ are the null geodesics $\gamma_z(r)$.

(ii) Suppose now that $\gamma(r)$ is causal and future-directed. We have 
\[
\p_r(\varphi(\gamma(r))) = d\varphi(\dot{\gamma}) = -g(Z, \dot{\gamma}).
\]
Since $g(Z,Z) = 0$ and $g(\dot{\gamma}, \dot{\gamma}) \leq 0$ and both $Z$ and $\dot{\gamma}$ are future-directed, it follows from Lemma \ref{lemma_timecone} that $\p_r(\varphi(\gamma(r))) \geq 0$ with equality if and only if $\dot{\gamma} = \lambda Z$ for some $\lambda > 0$. The past-directed case is analogous.
\end{proof}

We conclude this section by stating some topological properties of the functions $\varphi_{\omega}$.

\begin{Lemma} \label{lemma_eikonal_top}
Let $g$ satisfy \eqref{metric_a1}--\eqref{metric_a2}. Assume that $\varphi$ is a smooth solution of $g(d\varphi, d\varphi) = 0$ in $\mR^{1+n}$ with $\varphi = t-x \cdot \omega$ outside $\mR \times \ol{B}$. If $(t_j,x_j) \in \mR \times \ol{B}$ and $t_j \to \pm \infty$, then $\varphi(t_j,x_j) \to \pm \infty$. In particular $\varphi|_{\mR \times \ol{B}}: \mR \times \ol{B} \to \mR$ is proper and surjective.
\end{Lemma}
\begin{proof}
It is enough to consider $\omega = e_1$. By Proposition \ref{prop_eikonal} we know that $\Phi$ is a diffeomorphism $\Sigma_- \times \mR \to \mR^{1+n}$. (In the proof of the main results we could also obtain this conclusion from Theorem \ref{thm_geometric}, which would avoid using the direction (b) $\implies$ (a) in Proposition \ref{prop_eikonal}.) Thus for any $w \in \mR^{1+n}$ there is a unique pair $(\tilde{\kappa}(w), r(w)) \in \Sigma_- \times \mR$ such that $w = \gamma_{\tilde{\kappa}(w)}(r(w))$. Write $\kappa = \tilde{\kappa}|_{\Sigma_+}$. The condition $\varphi = t-x_1$ outside $\mR \times \ol{B}$ implies that for any $w \in \Sigma_+$, one has  
\[
\dot{\gamma}_{\kappa(w)}(r(w)) = -d\varphi(w)^{\sharp} = (1,e_1).
\]
Together with \eqref{metric_a1} these facts show that the family $\gamma_z$ satisfies (i)--(iii) in Lemma \ref{lemma_kappa}. Thus $\kappa: \Sigma_+ \to \Sigma_-$ is an embedding.

Next we prove that $\kappa$ is proper. Let $K \subset \Sigma_-$ be compact, so that for some $C > 1$ one has 
\[
K \subset \{ (t,-1,x') \,:\, |t| \leq C, \, |x'| \leq C \}.
\]
Let $w = (\tilde{t}, 1, \tilde{x}') \in \kappa^{-1}(K)$, so there is a unique $z = (t,-1,x') \in K$ with $\gamma_z(r_+(z)) = w$. The key fact that we need is that $\varphi$ is constant along the curves $\gamma_z$. Since also $\varphi = t-x_1$ outside $\mR \times \ol{B}$, we have 
\[
|\tilde{t} - 1| = |\varphi(w)| = |\varphi(z)| = |t+1|.
\]
This implies that 
\[
|\tilde{t}| \leq |t| + 2 \leq C + 2.
\]
On the other hand, if $|\tilde{x}'| > C$ then \eqref{metric_a1} implies that $|x'| > C$, so $w \notin \kappa^{-1}(K)$. This proves that 
\[
\kappa^{-1}(K) \subset \{ (\tilde{t},1,\tilde{x}') \,:\, |\tilde{t}| \leq C+2, \, |\tilde{x}'| \leq C \}.
\]
Thus $\kappa$ is proper, and Lemma \ref{lemma_kappa} implies that for any $z \in \Sigma_-$ the geodesic $\gamma_z$ reaches $\Sigma_+$ at a unique time $r_+(z)$. By (i) it also follows that $\gamma_z(r)$ stays outside $\mR \times \ol{B}$ unless $r \in [0,r_+(z)]$.

We now prove that $\varphi|_{\mR \times \ol{B}}$ is proper. Suppose that $w \in (\varphi|_{\mR \times \ol{B}})^{-1}([-C,C])$, so $w \in \mR \times \ol{B}$ and $|\varphi(w)| \leq C$. Write $\tilde{\kappa}(w) = (t,-1,x')$. Since $\varphi$ is constant along $\gamma_{\tilde{\kappa}(w)}$, we have 
\[
|t+1| = |\varphi(\tilde{\kappa}(w))| = |\varphi(w)| \leq C.
\]
Thus $|t| \leq C + 1$. Consider the set 
\[
S = \{ (t,-1,x') \in \Sigma_- \,:\, |t| \leq C+1, \, |x'| \leq 1 \}.
\]
Since $\gamma_z(r)$ stays outside $\mR \times \ol{B}$ unless $r \in [0,r_+(z)]$, we have 
\[
(\varphi|_{\mR \times \ol{B}})^{-1}([-C,C]) \subset \{ \Phi(z,r) \,:\, z \in S, \, r \in [0,r_+(z)] \}.
\]
Now $S$ is compact and $r_+$ is smooth, so the last set is contained in $\Phi(S \times [0,R])$ for some $R > 0$. This set is compact by continuity of $\Phi$. We have proved that $\varphi|_{\mR \times \ol{B}}$ is proper. It is also surjective onto $\mR$ since $\varphi = t - x \cdot \omega$ on $\mR \times \p B$.

Finally, fix $s \in \mR$. The set $\{ \varphi > s \}$ is connected, since any $w_j \in \{ \varphi > s \}$ for $j = 1, 2$ can be first connected to $\tilde{\kappa}(w_j)$ by the curve $\gamma_{\tilde{\kappa}(w_j)}$ along which $\varphi$ is constant, and then $\tilde{\kappa}(w_1)$ and $\tilde{\kappa}(w_2)$ can be joined by a line in $\Sigma_- \cap \{ \varphi > s \} = \Sigma_- \cap \{ t > s-1 \}$. Similarly $\{ \varphi < s \}$ is connected, and these are maximal connected sets in $\mR^{1+n} \setminus \{ \varphi = s \}$ since $\varphi$ is continuous. Using that $\varphi|_{\mR \times \ol{B}}$ is proper, the set $\{ \varphi = s \} \cap (\mR \times \ol{B})$ is compact and hence contained in $\{ |t| \leq C \}$ for some $C$. The connected set $\{ t > C \} \cap (\mR \times \ol{B})$ must be contained in the component $\{ \varphi > s \}$ since $\varphi = t - x_1$ on $\mR \times \p B$. It follows that 
\[
\{ \varphi \leq s \} \cap (\mR \times \ol{B}) \subset \{ t \leq C \}.
\]
This proves that $\varphi(t_j,x_j) \to +\infty$ whenever $(t_j,x_j) \in \mR \times \ol{B}$ and $t_j \to \infty$. The case where $t_j \to -\infty$ is analogous.
\end{proof}


\section{Plane waves} \label{sec_pw}

In this section we let $g$ be a smooth Lorentzian metric in $\mR^{1+n}$ satisfying \eqref{metric_a1}--\eqref{metric_a2}. We fix a Cauchy temporal function $\tau$ having the properties in Proposition \ref{prop_gh} (our constructions will be independent of the choice of $\tau$). The following result gives a precise definition and some properties of the distorted plane waves.

\begin{Proposition} \label{prop_pw}
Let $g$ satisfy \eqref{metric_a1}--\eqref{metric_a2}. For any $\omega \in S^{n-1}$ and $s \in \mR$ there is a unique solution $U = U_{\omega,s}^g \in L^2_{\mathrm{loc}}(\mR^{1+n})$ of the problem 
\[
\Box_g U = 0 \text{ in $\mR^{1+n}$}, \qquad U|_{\{ \tau < \tau_0 \}} = H(t-x \cdot \omega-s)
\]
where $\tau_0 \in \mR$ is chosen so that $\mathrm{supp}(\Box_g(H(t-x \cdot \omega-s))) \subset \{ \tau > \tau_0 \}$. The definition is independent of the choice of such $\tau_0$, and also independent of the choice of the Cauchy temporal function $\tau$.
Moreover, writing $\omega = e_1$ and $x = (x_1, x')$, one has 
    \begin{equation}\label{supp_U}
\mathrm{supp}(U) \subset J^+(\{ x_1 \leq -1 \text{ or } |x'| \geq 1 \} \cap \{  t \geq x_1 + s \}),
    \end{equation}
and $\WF(U)$ is the flowout of $\{ (z, \lambda(1,-\omega)) \,:\, z \in \Sigma_{-,s-1}, \, \lambda \neq 0 \}$ under the bicharacteristic flow for $\Box_g$.
\end{Proposition}

\begin{proof} 
By \eqref{metric_a1} the function $f := -\Box_g(H(t-x \cdot \omega-s))$ is supported in the set $\{ (t,x) \,:\, |x| \leq 1, \ t \geq x \cdot \omega + s \}$ and therefore also in $[s-1,\infty) \times \ol{B}$. By Proposition \ref{prop_tau_proper} (a) there is $\tau_0$ with $\supp(f) \subset \{ \tau > \tau_0 \}$. Since $g$ is globally hyperbolic, Proposition \ref{prop_wp} ensures that there is a unique solution $R := G_+ f$ of $\Box_g R = f$ with $R|_{\{ \tau < \tau_0\}} = 0$ and $\supp(R) \subset J^+(\supp(f))$. This shows  that there is a unique distributional solution 
\[
U = H(t-x \cdot \omega-s) -G_+(\Box_g(H(t-x \cdot \omega-s)))
\]
of $\Box_g U = 0$ in $\mR^{1+n}$ with $U|_{\{ \tau < \tau_0 \}} = H(t-x \cdot \omega-s)$.

Now let $\tilde{\tau}$ be another Cauchy temporal function and let $\tilde{\tau}_0$ be such that $\supp(f) \subset \{ \tilde{\tau} > \tilde{\tau}_0 \}$ (such a number exists by Proposition \ref{prop_tau_proper} (a)). Let $\tilde{R}$ be the solution of $\Box_g \tilde{R} = f$ with $\tilde{R}|_{\{ \tilde{\tau} < \tilde{\tau}_0 \}} = 0$. Then $\supp(\tilde{R}) \subset J^+(\supp(f))$, and consequently $\Box_g(R-\tilde{R}) = 0$ with $\supp(R-\tilde{R}) \subset J^+(\supp(f))$. Since $\tau$ is strictly increasing along future-directed causal curves, $J^+(\supp(f))$ is contained in $\{ \tau > \tau_0 \}$ (or similarly in $\{ \tilde{\tau} > \tilde{\tau}_0 \}$). Thus $R-\tilde{R}|_{\{ \tau < \tau_0 \}} = 0$, and uniqueness of the Cauchy problem implies that $R = \tilde{R}$ everywhere. This shows that the definition of $U$ is independent of the choice of Cauchy temporal function and of the choice of $\tau_0$.

To show \eqref{supp_U}, write $\tilde{U}_0 = H(t-x\cdot \omega - s) \chi_{\{ \tau < \tau_0 \}}$ where $\chi_{\{ \tau < \tau_0 \}}$ is the characteristic function of $\{ \tau < \tau_0 \}$. Then $U = \tilde{U}_0 + \tilde{R}$ where $\tilde{R}$ is the unique forward solution of 
\[
\Box_g \tilde{R} = -\Box_g \tilde{U}_0, \qquad \tilde{R}|_{\{ \tau < \tau_0 \}} = 0.
\]
Fix $T \geq 3$. By Proposition \ref{prop_tau_proper} (a) we may choose $\tau_0$ so that 
    \begin{equation}\label{tau0_bound}
\{ \tau \leq \tau_0 \} \subset \{ (t,x) \in \mR^{1+n} \,:\, t < s - (T-1) + \max( |x| -1, 0 ) \}.
    \end{equation}
Therefore $\supp(\tilde{U}_0) \subset A$, where 
    \begin{equation*}
A = \{ t < s-(T-1) + \max( |x| -1, 0 ) \} 
\cap \{ t \geq x \cdot \omega + s \}.
    \end{equation*}
It follows from $T > 2$ that $A$ is disjoint from $\{ |x| \leq 1 \}$. Writing $\omega = e_1$ and $x = (x_1,x')$ gives more precise bounds such as 
\[
A \subset  \{ x_1 \leq \frac{1}{2T}(|x'|^2-T^2) \}
\cap \{ t \geq x_1+s \},
\]
where $T$ can be chosen arbitrarily large. In particular, the condition $T \geq 3$ gives  
\begin{equation} \label{uzero_more_precise_bound}
A \subset \{ x_1 \leq -1 \text{ or } |x'| \geq 1 \} 
\cap \{ t \geq x_1+s \}.
\end{equation}
Since $\tilde{U}_0$ is supported in $A$, finite propagation speed gives that $\tilde{R}$ and hence also $U$ are supported in $J^+(A)$. Hence \eqref{supp_U} follows from \eqref{uzero_more_precise_bound}.

Let us next determine $\WF(U)$. Since $\Box_g U = 0$, it follows that $\WF(U)$ is contained in the characteristic set of $\Box_g$ \cite[Theorem 8.3.1]{hormander}, and by propagation of singularities $\WF(U)$ is invariant under the bicharacteristic flow there \cite[Theorem 26.1.1]{hormander}. Let $\gamma(r) = (z(r), \zeta(r))$ be any null bicharacteristic of $\Box_g$ and let $\varepsilon > 0$. By Lemma \ref{lemma_bichar_geodesic} $z(r)$ is a null geodesic, hence a causal curve, and therefore it meets $\{ \tau = \tau_0 - \varepsilon \}$ at exactly one time $r_0$. We also recall that $U = H(t-x \cdot \omega-s)$ in $\{ \tau < \tau_0 \}$. The fact that $\WF(U)$ is invariant under bicharacteristic flow yields that 
\[
(z(0), \zeta(0)) \in \WF(U) \ \  \Longleftrightarrow \ \  (z(r_0), \zeta(r_0)) \in \WF(H(t-x \cdot \omega-s)) = N^*(\{t = x \cdot \omega + s \}).
\]
Recalling \eqref{tau0_bound}, we have 
    \begin{equation*}
\{ \tau = \tau_0 - \varepsilon \} \cap \{t = x \cdot \omega + s \} \subset A.
    \end{equation*}
By \eqref{uzero_more_precise_bound} it follows (writing again $\omega = e_1$) that
\[
z(r_0) \in \{ x_1 \leq -1 \text{ or } |x'| \geq 1 \} \cap \{ t = x_1+s \}.
\]
If $z(r_0) \in \{ x_1 \geq -1, |x'| \geq 1 \}$, then by \eqref{metric_a1} the bicharacteristic must also contain a point of  $N^*( \{ x_1 \leq -1,\ t = x_1 + s \} )$. Thus $\WF(U)$ is the flowout of $N^*( \{ x \cdot \omega \leq -1,\ t = x \cdot \omega + s \} )$ under the bicharacteristic flow. By \eqref{metric_a1} this is also the flowout of $\{ (z, \lambda (1,-\omega)) \,:\, z \in \Sigma_{-,s-1}, \, \lambda \neq 0 \}$.
 
The function $H(t-x \cdot \omega - s)$ is in $L^2_{\mathrm{loc}}$ and therefore has empty $L^2$ wave front set. Thus repeating the above argument for the $L^2$ wave front set of $U$ (using the Sobolev wave front set results in \cite[Theorems 18.1.31 and 26.1.4]{hormander} instead of the $C^{\infty}$ versions) proves that the $L^2$ wave front set of $U$ is empty. Then \cite[Theorem 18.1.31]{hormander} implies that $U \in L^2_{\mathrm{loc}}(\mR^{1+n})$.
\end{proof}

If there is a smooth solution $\varphi_{\omega}$ of the eikonal equation, the plane wave $U_{\omega,s}$ has the following support properties.

\begin{Lemma} \label{lem_h_pwe}
Let $g$ satisfy \eqref{metric_a1}--\eqref{metric_a2} and suppose that $\varphi$ is a smooth solution of $g(d\varphi,d\varphi) = 0$ in $\mR^{1+n}$ with $\varphi|_{\{ x \cdot \omega \leq -1 \}} = t - x \cdot \omega$. Then $\supp(U_{\omega,s}^g) \subset \{ \varphi \geq s \}$ and $U_{\omega,s}^g$ is smooth in $\{ \varphi > s \}$.
\end{Lemma}

\begin{proof}
Write $\omega = e_1$ and $x = (x_1,x')$. Since $\varphi|_{\{ x_1 \leq -1 \}} = t-x_1$ and $\varphi$ is constant along the curves $\gamma_z$, it follows from \eqref{metric_a1} that $\varphi = t-x_1$ in $\{ x_1 \leq -1 \text{ or } |x'| \geq 1 \}$. From Proposition \ref{prop_pw} we know that $U = U_{\omega,s}^g$ is in $L^2_{\mathrm{loc}}(\mR^{1+n})$ and that 
\[
\supp(U) \subset J^+(\{ x_1 \leq -1 \text{ or } |x'| \geq 1 \} \cap \{ t \geq x_1 + s \}) \subset J^+(\{ \varphi \geq s \}).
\]
Lemma \ref{lemma_eikonal} implies that $\varphi$ is nondecreasing along any future-directed causal curve. This implies that 
\[
\supp(U) \subset \{ \varphi \geq s \}.
\]
By Proposition \ref{prop_pw}, $\WF(U)$ lies over $\{ \varphi = s \}$ so $U$ is smooth in $\{ \varphi > s \}$.
\end{proof}

Lemma \ref{lem_h_pwe} would already be sufficient for the applications to inverse problems. In the next result we give more precise information and record the basic fact that the plane waves are conormal distributions whenever the eikonal equation has a smooth solution. See \cite[Section 18.2]{hormander} for properties of conormal distributions and the notation $I^m(X,Y)$.

\begin{Lemma} \label{lem_pw_conormal}
Assume the conditions in Lemma \ref{lem_h_pwe}. Then $U = U_{\omega,s}^g \in I^{-\frac{1+n}{4}}(\mR^{1+n}, Y)$ where $Y = \{ \varphi = s \}$, and one has the representation 
\[
U_{\omega,s} = u H(\varphi-s)
\]
where $u$ is smooth in $\{ \varphi \geq s \}$.
\end{Lemma}

The fact that $U$ is conormal follows by writing $U = \hat{U}_0 + \hat{R}$ where $\hat{U}_0 = H(t-x \cdot \omega -s) \psi(\tau-\tau_0)$ with $\psi \in C^{\infty}(\mR)$ satisfying $\psi(t) = 1$ for $t < 0$ and $\psi(t) = 0$ for $t > \delta$ with $\delta > 0$ small, and letting $\hat{R}$ be the solution of 
\[
\Box_g \hat{R} = -\Box_g \hat{U}_0, \qquad R|_{\{ \tau < \tau_0 \}} = 0
\]
with $\tau_0$ chosen as in \eqref{uzero_more_precise_bound}. Then $\hat{U}_0$ is conormal with respect to $\{ \varphi=s \}$, and \cite[Proposition 2.3]{greenleaf1993} with the choice $\Lambda_1 = N^*(\{ \varphi = s \})$ ensures that $\hat{R}$ and hence $U$ are conormal with respect to $\{ \varphi = s \}$. For completeness we give another direct proof of Lemma \ref{lem_pw_conormal} in Appendix \ref{sec_app}.


\section{Distinguishing a Lorentzian metric from the Minkowski metric} \label{sec_rigidity_proofs}

In this section we will prove our main results, beginning with Theorem \ref{thm_rigidity}. We assume that we can measure $U_{\omega,s}|_{\mR \times \p B}$ for all $\omega \in \Omega$ and for all time-delay parameters $s \in \mR$, where $\Omega$ is the finite set 
\[
\Omega = \{ \pm e_j, \frac{1}{\sqrt{2}}(e_j + e_k) \,:\, 1 \leq j, k \leq n, \ j < k \}.
\]
More precisely, recall the forward operator 
\[
\mathcal{F}_{\Omega}: g \mapsto (U_{\omega,s}^g|_{\mR \times \p B})_{\omega \in \Omega, s \in \mR}.
\]
We assume that 
\[
\mathcal{F}_{\Omega}(g) = \mathcal{F}_{\Omega}(\gm).
\]
We wish to prove that $F^* g = \gm$ for some diffeomorphism $F: \mR^{1+n} \to \mR^{1+n}$ with $F = \mathrm{id}$ outside $\mR \times \ol{B}$.

We now begin the proof of the uniqueness result. Let $\omega \in \Omega$ and $s \in \mR$. First we use well-posedness of the exterior Dirichlet problem to prove the following.

\begin{Lemma} \label{lem_h_exterior}
Let $g$ satisfy \eqref{metric_a1}--\eqref{metric_a2}. If $U_{\omega,s}^g|_{\mR \times \p B} = U_{\omega,s}^{\gm}|_{\mR \times \p B}$, then 
\[
U_{\omega,s}^g = U_{\omega,s}^{\gm} \text{ in $\mR \times (\mR^n \setminus \ol{B})$.}
\]
\end{Lemma}
\begin{proof}
Write $V = U_{\omega,s}^g - U_{\omega,s}^{\gm}$. Then $V$ is an $L^2_{\mathrm{loc}}$ function in $\mR^{1+n}$ that satisfies 
\[
\Box_{\gm} V = 0 \text{ in $\mR \times (\mR^n \setminus \ol{B})$}, \qquad V|_{\mR \times \p B} = 0, \qquad V|_{\{\tau < \tau_0 \}} = 0.
\]
We wish to use propagation of singularities in the exterior region $X := \mR \times (\mR^n \setminus B)$ to conclude that $V$ is smooth in $X$. Indeed, any compressed generalized bicharacteristic of $\Box_{\gm}$ in $X$ reaches the Cauchy surface $\{ \tau = \tau_0 - \varepsilon\}$, near which $V$ vanishes, after at most one reflection on $\p B$ that is either transversal or diffractive. By propagation of singularities (see \cite[Theorem 7.4.3]{melrose2024} or \cite[Theorem 24.5.3]{hormander}) we have $\mathrm{WF}_b(V|_X) = \emptyset$, where $V|_X$ is defined by \cite[Corollary 18.3.31]{hormander}. By \cite[Theorem 18.3.27]{hormander} we have $V|_{X} \in C^{\infty}(X)$. Now that we have proved that $V$ is a smooth solution in $X$, uniqueness \cite[Theorem 24.1.1]{hormander} (or more precisely iteration of \cite[Theorems 23.2.7 and 24.1.4]{hormander}) implies that $V$ must vanish in $\mR \times (\mR^n \setminus \ol{B})$.
\end{proof}

At this point we can prove Theorems \ref{thm_pw_spw_scattering} and \ref{thm_eikonal_pw}.

\begin{proof}[Proof of Theorem \ref{thm_pw_spw_scattering}]
If $\mathcal{F}_{\omega}(g) = \mathcal{F}_{\omega}(\gm)$, Lemma \ref{lem_h_exterior} implies that 
\[
\WF(U_{\omega,s}^g|_{\{ |x| > 1 \}}) = \WF(U_{\omega,s}^{\gm}|_{\{ |x| > 1 \}}) = N^*(\{ t = x \cdot \omega + s \}) \cap T^*(\{ |x| > 1 \}).
\]
By propagation of singularities we also know that $\WF(U_{\omega,s}^g)$ is contained in the characteristic set of $\Box_g$ and it is invariant under the bicharacteristic flow there. The bicharacteristic curves are related to null geodesics as described in Lemma \ref{lemma_bichar_geodesic}. Combining all these facts with \eqref{metric_a1}, we see that
\begin{equation} \label{u_wf_pm}
\WF(U_{\omega,s}^g) \cap \pi^{-1}(\Sigma_{\pm}) = \{ (z, \lambda (1,-\omega)) \,:\, z \in \Sigma_{\pm, s \pm 1}, \, \lambda \neq 0 \}
\end{equation}
where $\pi: T^* \mR^{1+n} \to \mR^{1+n}$ is the natural projection.

On the other hand, $\WF(U_{\omega,s}^g)$ is the flowout of  $\{ (z, \lambda(1,-\omega)) \,:\, z \in \Sigma_{-, s - 1}, \, \lambda \neq 0 \}$ under the bicharacteristic flow for $\Box_g$ by Proposition \ref{prop_pw}. Using Lemma \ref{lemma_bichar_geodesic} again, it follows that 
\[
\WF(U_{\omega,s}^g) \cap \pi^{-1}(\Sigma_{+}) = \{ (\gamma_z(r), -\lambda \dot{\gamma}_z(r)^{\flat}) \,:\, z \in \Sigma_{-,s-1}, \, \gamma_z(r) \in \Sigma_+, \, \lambda \neq 0  \}.
\]
Comparing these two expressions for $\WF(U_{\omega,s}^g) \cap \pi^{-1}(\Sigma_{+})$, we see that any $\gamma_z$ that reaches $\Sigma_+$ must stay in $\{ x \cdot \omega > 1 \}$ afterwards and that 
\begin{align}
\beta_g(\Sigma_{-,s-1} \setminus \mathcal{T}_g) &= \{ (\gamma_z(r_+(z)), \dot{\gamma}_z(r_+(z))) \,:\, z \in \Sigma_{-,s-1} \setminus \mathcal{T}_g \} \label{beta_wf_eq} \\
 &= \{ (\tilde{z}, \lambda(\tilde{z})(1,\omega)) \,:\, \tilde{z} \in \Sigma_{+, s+1} \}, \notag
\end{align}
where $\lambda$ is some nowhere vanishing function on $\Sigma_+$.

One has $\lambda > 0$ since $\dot{\gamma}_z(r_+(z)))$ cannot point into $\{ x \cdot \omega < 1 \}$. Next we show that $\lambda$ is smooth. For $w \in \Sigma_+$, let $\eta_w$ be the $g$-geodesic with $\eta_w(0)=0$ and $\dot{\eta}_w(0) = (1,\omega)$. By \eqref{beta_wf_eq} any $\bar{w} \in \Sigma_+$ is of the form $\bar{w} = \gamma_{\bar{z}}(r_+(\bar{z}))$, and $\eta_{\bar{w}}$ meets $\Sigma_-$ transversally since $\gamma_{\bar{z}}$ does. By Lemma \ref{lemma_intersection} there is a smooth function $R$ such that $z(w) = \eta_w(-R(w)) \in \Sigma_-$ for $w$ near $\bar{w}$. Thus $w = \gamma_{z(w)}(r_+(z(w)))$ near $\bar{w}$. Since $\gamma_{\bar{z}}$ meets $\Sigma_+$ transversally, also $r_+$ is smooth near $\bar{z}$ by Lemma \ref{lemma_intersection}. We obtain the formula 
\[
\lambda(w) = \frac{1}{\sqrt{2}} |\dot{\gamma}_{z(w)}(r_+(z(w)))|_{\mathrm{Eucl}}.
\]
Thus $\lambda$ is smooth on $\Sigma_+$. This concludes the proof that $\beta_g(\Sigma_{-,s-1} \setminus \mathcal{T}_g) = \beta_{\gm}^{\lambda}(\Sigma_{-,s-1})$ where $\lambda \in C^{\infty}(\Sigma_+)$ is positive.
\end{proof}

\begin{proof}[Proof of Theorem \ref{thm_eikonal_pw}]
Write $\omega = e_1$. The conclusion of Theorem \ref{thm_geometric} is that 
\begin{align*}
\text{$\Phi: \Sigma_- \times \mR \to \mR^{1+n}$ is a diffeomorphism with $\dot{\gamma}_z(r) = \hat{\lambda}(1,e_1)$ when $\gamma_z(r) \notin \mR \times \ol{B}$}.
\end{align*}
Here $\hat{\lambda} \in C^{\infty}(\mR^{1+n})$ is positive and satisfies $\hat{\lambda} = 1$ in $\{ x_1 \leq -1 \}$. By \eqref{metric_a1}, we also have $\hat{\lambda} = 1$ in $\{ |y| \geq 1 \}$ if we write $z = (t, x_1, y)$. By Proposition \ref{prop_eikonal}, the fact that $\Phi$ is a diffeomorphism implies the existence of smooth solution $\varphi$ to the eikonal equation
\[
g(d\varphi, d\varphi) = 0, \qquad \varphi|_{\{ x_1 \leq -1 \}} = t-x_1.
\]
Moreover, when $\gamma_z(r) \notin \mR \times \ol{B}$ we have 
\[
d\varphi(\gamma_z(r)) = -\dot{\gamma}_z(r)^{\flat} = \hat{\lambda}(\gamma_z(r)) (1, -e_1).
\]
If we define $q(t,x) = t - x_1$, this can be written as 
\[
d \varphi = \hat{\lambda} \,d q \quad \text{ outside $\mR \times \ol{B}$}.
\]
Thus, outside $\mR \times \ol{B}$ we have 
\[
0 = d(d\varphi) = d(\hat{\lambda} \,dq) = d\hat{\lambda} \wedge dq.
\]
If we consider coordinates $(p,q,y)$ where $p = t+x_1$, this gives 
\[
\p_p \hat{\lambda} \,dp \wedge dq + \sum \p_{y_j} \hat{\lambda} \,dy_j \wedge dq = 0 \text{ outside $\mR \times \ol{B}$.}
\]
By linear independence $\p_p \hat{\lambda} = \p_{y_j} \hat{\lambda}  = 0$ outside $\mR \times \ol{B}$. Combining this with the condition $\hat{\lambda}|_{\{|y| \geq 1 \}} = 1$, we obtain $\hat{\lambda} = 1$ in $\{ x_1 \geq 1 \}$. Since $\hat{\lambda} = 1$ also in $\{ x_1 \leq -1 \}$, the condition \eqref{metric_a1} ensures that $\hat{\lambda} = 1$ outside $\mR \times \ol{B}$. Thus $d\varphi = dq$ outside $\mR \times \ol{B}$, so $\varphi - q$ is constant outside $\mR \times \ol{B}$, and this constant is $0$ by the condition $\varphi|_{\{ x_1 \leq -1 \}} = t-x_1$.

We have proved the claims about $\varphi$. Lemma \ref{lem_pw_conormal} implies that the plane waves have the required representation 
\[
U_{\omega,s}^g(t,x) = u_{\omega,s}(t,x) H(\varphi_{\omega}(t,x)-s) \text{ in $\mR^{1+n}$}
\]
where $u_{\omega,s}$ is smooth in $\{ \varphi_{\omega} \geq  s \}$.
\end{proof}

\begin{Remark}
The fact that $\hat{\lambda} \equiv 1$ would also follow from the conclusion in Proposition \ref{prop_nontrapping} that $z \mapsto \gamma_z(r_+(z))$ maps $\Sigma_{-,\sigma}$ onto $\Sigma_{+,\sigma+2}$ for any $\sigma \in \mR$. Combining that with the facts that $\varphi$ is constant along $\gamma_z$ and $\varphi|_{\{x_1 \leq -1\}} = t-x_1$ would imply that $\varphi|_{\{x_1 \geq 1 \}} = t-x_1$.
\end{Remark}

In particular, if $\mathcal{F}_{\Omega}(g) = \mathcal{F}_{\Omega}(\gm)$, then Theorems \ref{thm_pw_spw_scattering}--\ref{thm_eikonal_pw} imply the following statement:
\begin{equation} \label{dp}
\left\{ \begin{array}{c} \text{For any $\omega \in \Omega$ the eikonal equation $g(d\varphi_{\omega}, d\varphi_{\omega}) = 0$ has a unique smooth} \\[5pt]  
\text{solution in $\mR^{1+n}$ with $\varphi_{\omega} = t - x \cdot \omega$ outside $\mR \times \ol{B}$.}
\end{array} \right.
\end{equation}

We now invoke for the first time the unique continuation property \eqref{metric_a3}. This is needed to prove the crucial fact that the functions $\varphi_{\omega}$ also solve the wave equation.

\begin{Lemma} \label{lemma_varphi_wave}
Let  $g$ satisfy \eqref{metric_a1}--\eqref{metric_a3} and let $\varphi_{\omega}$ be as in \eqref{dp}. Suppose that for some fixed $s$ one has $U_{\omega,s}^g|_{\mR \times \p B} = U_{\omega,s}^{\gm}|_{\mR \times \p B}$. Then 
\[
\Box_g \varphi_{\omega} = 0 \text{ in $\{ (t,x) \in \mR^{1+n} \,:\, \varphi_{\omega}(t,x) = s \}$}.
\]
Moreover, one has 
\[
U_{\omega,s}^g = H(\varphi_{\omega}(t,x)-s) \text{ in $\mR^{1+n}$.}
\]
\end{Lemma}
\begin{proof}
Define $v = u_{\omega,s}-1$ in $Q_+$, where 
\begin{align*}
Q_+ &= \{ (t,x) \,:\, \varphi_{\omega} > s, \ x \in B \}, \\
\Gamma_+ &= \{ (t,x) \,:\, \varphi_{\omega} > s, \ x \in \p B \}.
\end{align*}
Then $v \in C^{\infty}(\ol{Q}_+)$ satisfies 
\begin{equation} \label{w_ucp_cond}
\Box_g v = 0 \text{ in $Q_+$}, \qquad v|_{\Gamma_+} = \p_{\nu} v|_{\Gamma_+} = 0.
\end{equation}
Lemma \ref{lemma_eikonal_top} guarantees that $(r,\infty) \times \ol{B} \subset \{ \varphi_{\omega} > s \}$ for some $r$. By the unique continuation property \eqref{metric_a3} there is $r_1$ such that $v = 0$ when $t > r_1$. Next we consider the zero extension of $v$ defined by 
\[
\tilde{v} = \left\{ \begin{array}{cl} v, & (t,x) \in Q_+, \\[3pt] 0, & \text{otherwise}. \end{array} \right.
\]
Then $\tilde{v} \in L^2_{\mathrm{loc}}(\mR^{1+n})$ and $\tilde{v}$ is $H^2$ across $\Gamma_+$ by \eqref{w_ucp_cond}. Writing $f = \Box_g \tilde{v}$, it follows that 
\[
\Box_g \tilde{v} = f \text{ in $\mR^{1+n}$},
\]
where $\supp(f) \subset \{ \varphi_{\omega} \leq s \}$ by \eqref{w_ucp_cond}. Moreover, $\tilde{v} = 0$ when $t > r_1$. By Proposition \ref{prop_tau_proper} (a) there is a Cauchy surface $\{ \tau = \tau_0 \}$ such that $\{ \tau \geq \tau_0 \} \cap (\mR \times \ol{B}) \subset \{ t > r_1 \}$. In particular, $\tilde{v}|_{\{ \tau > \tau_0 \}} = 0$. Then uniqueness in the backward Cauchy problem and finite propagation speed in Proposition \ref{prop_wp} imply that $\supp(\tilde{v}) \subset J^{-}(\supp(f)) \subset J^{-}(\{ \varphi_{\omega} \leq s \})$. Since by Lemma \ref{lemma_eikonal} the function $\varphi_{\omega}$ is nonincreasing along past-directed causal curves, we get $\supp(\tilde{v}) \subset \{ \varphi_{\omega} \leq s \}$. This implies that $u_{\omega,s} = 1$ in $Q_+$.

In particular, we have 
\[
U_{\omega,s}^g = H(\varphi_{\omega}-s) \text{ in $\mR^{1+n}$}.
\]
Then $\Box_g(H(\varphi_{\omega}-s)) = 0$ in $\mR^{1+n}$. A direct computation gives 
\[
\Box_g(H(\varphi_{\omega}-s)) = -g(d\varphi_{\omega},d\varphi_{\omega}) H''(\varphi_{\omega}-s) + (\Box_g \varphi_{\omega}) H'(\varphi_{\omega}-s).
\]
Since $g(d\varphi_{\omega},d\varphi_{\omega}) = 0$, it follows that 
\[
\Box_g \varphi_{\omega}|_{\{\varphi_{\omega}=s\} } = 0. \qedhere
\]
\end{proof}

In the remainder of this section we assume that $g$ satisfies \eqref{metric_a1}--\eqref{metric_a3}, that $\mathcal{F}_{\Omega}(g) = \mathcal{F}_{\Omega}(\gm)$, and that for any $\omega \in \Omega$, $\varphi_{\omega}$ is as in \eqref{dp}. We have proved that 
\[
\Box_g \varphi_{\omega}|_{ \{ \varphi_{\omega} = s \} } = 0.
\]
Since we have measurements for all $s \in \mR$ and $\omega \in \Omega$, we in fact have 
\[
\Box_g \varphi_{\omega} = 0 \text{ in $\mR^{1+n}$ for any $\omega \in \Omega$.}
\]

In the Minkowski case $\varphi_{\omega} = t - x \cdot \omega$ everywhere. We wish to have a similar representation in the general case. Define 
\begin{align*}
\sigma &= \frac{1}{2} (\varphi_{e_1} + \varphi_{-e_1}), \\
\alpha_j &= \frac{1}{2} (-\varphi_{e_j} + \varphi_{-e_j}).
\end{align*}
Note that $\Box_g \sigma = \Box_g \alpha_j = 0$ in $\mR^{1+n}$, and $\sigma = t$ and $\alpha_j = x_j$ when $x \notin \ol{B}$.

\begin{Lemma} \label{lemma_aom}
For any $\omega \in \Omega$ one has 
\[
\varphi_{\omega} = \sigma - \sum \omega_j \alpha_j.
\]
\end{Lemma}
\begin{proof}
Write $v = \varphi_{\omega} - (\sigma - \sum \omega_j \alpha_j)$. Then 
\[
\Box_g v = 0 \text{ in $\mR^{1+n}$}, \qquad v = 0 \text{ when $x \notin \ol{B}$.}
\]
By the unique continuation property \eqref{metric_a3}, it follows that $v = 0$ when $t > T$ for some $T$. Proposition \ref{prop_tau_proper} (a) ensures that there is a Cauchy surface $\{\tau=\tau_0\}$ with $\{ \tau \geq \tau_0 \} \cap (\mR \times \ol{B}) \subset \{ t > T \}$. Uniqueness in the backward Cauchy problem then implies that $v = 0$ in $\mR^{1+n}$.
\end{proof}

We now use Lemma \ref{lemma_aom} to show the following.

\begin{Lemma} \label{lemma_h_diagonal}
One has 
\begin{align*}
g(d\sigma,d\sigma) &= -\gamma(t,x), \\
g(d\sigma,d\alpha_j) &= 0, \\
g(d\alpha_j, d\alpha_k) &= \gamma(t,x) \delta_{jk},
\end{align*}
where $\gamma > 0$ is smooth and $\gamma = 1$ for $x \notin \ol{B}$.
\end{Lemma}
\begin{proof}
For $\omega \in \Omega$  we have 
\begin{align*}
0 &= g(d \varphi_{\omega}, d\varphi_{\omega}) = g(d\sigma-\omega_j d\alpha_{j}, d\sigma-\omega_k d\alpha_{k}) \\
 &= g(d\sigma,d\sigma) - 2 \omega_j g(d\sigma, d\alpha_{j}) + g(d\alpha_{j}, d\alpha_{k}) \omega_j \omega_k.
\end{align*}
In other words 
\[
[g(d\sigma,d\sigma)\delta_{jk} + g(d\alpha_{j}, d\alpha_{k})] \omega_j \omega_k - 2 \omega_j g(d\sigma, d\alpha_{j}) = 0, \qquad \omega \in \Omega.
\]
Choosing $\omega = \pm e_j$ and subtracting the resulting equations, we get 
\[
g(d\sigma, d\alpha_{j}) = 0, \qquad 1 \leq j \leq n.
\]
Then, writing $\gamma(t,x) := -g(d\sigma, d\sigma)$ and choosing $\omega=e_j$, we obtain 
\begin{align*}
g(d\alpha_j, d\alpha_j) = \gamma, \qquad 1 \leq j \leq n.
\end{align*}
Finally, choosing $\omega = \frac{1}{\sqrt{2}}(e_j + e_k)$ with $j \neq k$, we obtain 
\begin{align*}
g(d\alpha_j, d\alpha_k) = 0, \qquad j \neq k.
\end{align*}

It remains to show that the conformal factor $\gamma$ can never vanish. We argue by contradiction and suppose that $\gamma(t_0,x_0) = 0$ for some $(t_0,x_0)$. We first observe that $\sigma = \frac{1}{2}(\varphi_{e_1} + \varphi_{-e_1})$ satisfies 
\[
4 g(d\sigma, d\sigma) = g(d\varphi_{e_1}, d\varphi_{e_1}) + 2 g(d\varphi_{e_1}, d\varphi_{-e_1}) + g(d\varphi_{-e_1}, d\varphi_{-e_1}).
\]
Since $\gamma(t_0,x_0) = 0$ and since $d\varphi_{\pm e_1}$ are null, at $(t_0,x_0)$ one must have 
\[
g(d\varphi_{e_1}, d\varphi_{-e_1}) = 0.
\]
By Lemma \ref{lemma_timecone}, any null future-directed covectors $d\varphi_{\pm e_1}$ satisfying the above condition must also satisfy $d\varphi_{-e_1} = \lambda d\varphi_{e_1}$ for some $\lambda > 0$. Then Proposition \ref{prop_eikonal}   gives that $e_1 = -e_1$, which is a contradiction. We have proved that $\gamma$ is nonvanishing, and by connectedness $\gamma > 0$ since $\gamma = 1$ outside $\ol{B}$.
\end{proof}

Write 
\[
F: \mR^{1+n} \to \mR^{1+n}, \ \ F(t,x) = (\sigma(t,x), \alpha_1(t,x), \ldots, \alpha_{n}(t,x)).
\]
If $x \notin \ol{B}$ one has $F(t,x) = (t, x)$. The result of Lemma \ref{lemma_h_diagonal} may be rewritten in matrix form as 
\begin{equation} \label{g_conformal}
(DF) g^{-1} (DF)^t = \gamma \gm^{-1}.
\end{equation}

\begin{Lemma} \label{lemma_f_diffeo}
$F$ is a diffeomorphism.
\end{Lemma}
\begin{proof}
By taking determinants in \eqref{g_conformal} and using that $\gamma > 0$ we see that $DF$ is invertible everywhere. By the Hadamard global inverse function theorem (Theorem \ref{thm_hadamard}) it is enough to show that $F: \mR^{1+n} \to \mR^{1+n}$ is proper. To prove this, let $K \subset \mR^{1+n}$ be compact. Then $F^{-1}(K)$ is closed and we need to show that $F^{-1}(K)$ is bounded. Now $F^{-1}(K) \cap (\mR \times \ol{B}^c) = K \cap (\mR \times \ol{B}^c)$ is bounded, so it is sufficient to show that $F^{-1}(K) \cap (\mR \times \ol{B})$ is bounded.

We argue by contradiction and suppose that there is a sequence $(t_j, x_j) \in F^{-1}(K) \cap (\mR \times \ol{B})$ with $t_j \to \infty$ (the case where $t_j \to -\infty$ is analogous). Since $K$ is compact one has  $|\sigma(t_j,x_j)| \leq C$ and $|\alpha_1(t_j,x_j)| \leq C$ for all $j$. Moreover, since $\varphi_{e_1} = \sigma - \alpha_1$, we have 
\[
|\varphi_{e_1}(t_j,x_j)| \leq 2 C.
\]
By Lemma \ref{lemma_eikonal_top} the map $\varphi_{e_1}|_{\mR \times \ol{B}}$ is proper. This implies that $|t_j| \leq C'$ for some $C'$, which is a contradiction.
\end{proof}

The equation \eqref{g_conformal} can be rewritten as 
\[
(F^{-1})^* g = \kappa \gm
\]
for the positive function $\kappa = \gamma^{-1}$ with $\kappa = 1$ for $x \notin \ol{B}$. Since $F$ is the identity for $x \notin \ol{B}$, the boundary measurements for $g$ and $(F^{-1})^* g$ agree. The proof is thus concluded by the following result.

\begin{Proposition} \label{prop_conformal_factor}
Let $g$ satisfy \eqref{metric_a1}--\eqref{metric_a3}, let $n \geq 2$, and let $\kappa > 0$ be smooth with $\kappa = 1$ outside $\mR \times \ol{B}$. If 
\[
\mathcal{F}_{e_1}(\kappa \gm) = \mathcal{F}_{e_1}(\gm),
\]
then $\kappa = 1$.
\end{Proposition}
\begin{proof}
We repeat the argument above with the choice $g = \kappa \gm$. By Lemma \ref{lem_h_exterior} we have $U_{e_1,s}^{\kappa \gm} = U_{e_1,s}^{\gm}$ outside $\mR \times \ol{B}$ for any $s \in \mR$. Since $g = \kappa \gm$, the eikonal equation for $g$ is the same as for $\gm$, and thus the function 
\[
\varphi(t,x) = t-x_1
\]
is a global smooth solution of $g(d\varphi,d\varphi) = 0$. Lemma \ref{lemma_varphi_wave} then ensures that 
\[
\Box_{\kappa \gm} (t-x_1) = 0 \text{ in $\mR^{n+1}$}.
\]
This equation can be rewritten as 
\[
(\p_t + \p_{x_1}) (\kappa^{\frac{n-1}{2}}) = 0.
\]
Since $\kappa = 1$ when $x_1 = -1$, this transport equation gives that $\kappa = 1$ everywhere.
\end{proof}

\begin{Remark} \label{rmk_onedim}
If $n = 1$, Proposition \ref{prop_conformal_factor} fails since $\kappa^{\frac{n-1}{2}} \equiv 1$. Thus we can only conclude that $(F^{-1})^* g = \kappa \gm$ where $F = \mathrm{id}$ and $\kappa = 1$ outside $\mR \times \ol{B}$. This is the best possible result for $n=1$, since in this case one has $\Box_{\kappa g} = \kappa^{-1} \Box_g$ and therefore also the boundary measurements have a conformal invariance.
\end{Remark}

\begin{proof}[Proof of Theorem \ref{thm_rigidity}]
Combining Theorems \ref{thm_pw_spw_scattering}--\ref{thm_eikonal_pw} with Lemmas \ref{lemma_aom}--\ref{lemma_f_diffeo} and Proposition \ref{prop_conformal_factor} yields Theorem \ref{thm_rigidity}.
\end{proof}

\begin{proof}[Proof of Corollary \ref{cor_compact_perturbation}]
If $g = \gm$ outside $(-T,T) \times \ol{B}$ and if $I = (r,\infty)$ where $r > T$, then any solution of 
\[
\Box_g u = 0 \text{ in $I \times B$}, \qquad u|_{I \times \p B} = \p_{\nu} u|_{I \times \p B} = 0
\]
must satisfy $u = 0$ in $(r+1, \infty) \times B$ by the Holmgren-John uniqueness theorem or by propagation of analytic singularities (see \cite[Theorems 8.6.5 and 8.6.13]{hormander}). Thus the unique continuation property \eqref{metric_a3} is valid. Theorem \ref{thm_rigidity} shows that $F^* g = \gm$ for some diffeomorphism with $F = \mathrm{id}$ outside $\mR \times \ol{B}$. Finally, since $g = \gm$ outside $(-T,T) \times \ol{B}$, the eikonal solutions satisfy $\varphi_{\omega} = t - x \cdot \omega$ for $|t| \geq T+2$. From the definitions of $\sigma$ and $\alpha_j$ we see that $F = \mathrm{id}$ when $|t| \geq T+2$.
\end{proof}

\begin{proof}[Proof of Corollary \ref{cor_dnmap}]
By Proposition \ref{prop_tau_proper}, the map $\tau|_{\mR \times \ol{B}}$ is proper and surjective. Therefore the Dirichlet problem on $\mR \times B$ with zero Cauchy data for $\tau \ll 0$ is well-posed by \cite[Theorem 24.1.1]{hormander}, and the map $\Lambda_g^{\mathrm{Hyp}}$ is well defined on $C^{\infty}_c(\mR \times \p B)$. Assume that $\Lambda_g^{\mathrm{Hyp}} = \Lambda_{\gm}^{\mathrm{Hyp}}$. We wish to show that $U_{\omega,s}^g|_{\mR \times \p B} = U_{\omega,s}^{\gm}|_{\mR \times \p B}$ for all $\omega \in S^{n-1}$ and $s \in \mR$. This can be done by approximating the boundary values of plane waves by smooth functions.

Choose $\varphi_j \in C^{\infty}_c(\mR)$ so that $\supp(\varphi_j) \subset [-1,\infty)$ and $\varphi_j \to H$ in the sense of distributions, and let $u_j^{\gm} = \varphi_j(t-x\cdot \omega - s)$. Then $u_j^{\gm}$ solves $\Box_{\gm} u_j^{\gm} = 0$ in $\mR^{1+n}$, $\supp(u_j^{\gm}) \subset \{ t - x \cdot \omega \geq s -1 \}$, and $u_j^{\gm} \to U_{\omega,s}^{\gm} = H(t - x \cdot \omega - s)$ in the sense of distributions. Write $f_j = u_j^{\gm}|_{\mR \times \p B}$, so $\supp(f_j) \subset [s-2,\infty) \times \p B$, and let $\hat{u}_j \in C^{\infty}(\mR \times \ol{B})$ be the solution of  
\[
\Box_g \hat{u}_j = 0 \text{ in $\mR \times B$}, \qquad \hat{u}_j|_{\mR \times \p B} = f_j, \qquad \hat{u}_j|_{\{ \tau < \tau_0 \} \cap (\mR \times \ol{B})} = 0,
\]
where $\tau_0$ is chosen using Proposition \ref{prop_tau_proper} (a) such that 
\[
[s-2, \infty) \times \ol{B} \subset \{ \tau > \tau_0 \}.
\]
By the assumption $\Lambda_g^{\mathrm{Hyp}} = \Lambda_{\gm}^{\mathrm{Hyp}}$, the Cauchy data of $\hat{u}_j$ and $ u_j^{\gm}$ agree on $\mR \times \p B$. We define 
\[
u_j = \left\{ \begin{array}{cl} \hat{u}_j & \text{in $\mR \times B$}, \\ u_j^{\gm} & \text{outside $\mR \times B$}. \end{array} \right.
\]
Then $u_j \in H^2_{\mathrm{loc}}(\mR^{1+n})$, and therefore $u_j$ solves $\Box_g u_j = 0$ in $\mR^{1+n}$.

By the condition $[s-2, \infty) \times \ol{B} \subset \{ \tau > \tau_0 \}$, one has $u_j|_{\{ \tau < \tau_0 \}  \cap (\mR \times \ol{B})} = u_j^{\gm}|_{\{ \tau < \tau_0 \}  \cap (\mR \times \ol{B})} = 0$, and by definition $u_j$ and $u_j^{\gm}$ agree in $\{\tau < \tau_0 \} \cap (\mR \times B^c)$. Since $\Box_g u_j = 0$ in $\mR^{1+n}$ and $u_j|_{\{ \tau < \tau_0 \}} = u_j^{\gm}|_{\{ \tau < \tau_0 \}}$, uniqueness of the Cauchy problem yields that 
\[
u_j = u_j^{\gm} - G_+(\Box_g(u_j^{\gm}))
\]
where $G_+$ is the forward solution operator for Cauchy surface $\{ \tau = \tau_0 \}$. Since $u_j^{\gm} \to U_{\omega,s}^{\gm}$ in the sense of distributions and $G_+$ is continuous \cite[Lemma 4.1]{bar2015}, we also have 
\[
u_j \to U_{\omega,s}^{\gm} - G_+(\Box_g(U_{\omega,s}^{\gm})) = U_{\omega,s}^g \text{ in the sense of distributions.}
\]
By taking distributional traces on $\mR \times \p B$ using \cite[Theorem 8.4.2 and Corollary 8.2.7]{hormander}, which is possible since the wave front sets of the solutions $u_j$ and $U_{\omega,s}^g$ are disjoint from $N^*(\mR \times \p B)$, we obtain 
\[
U_{\omega,s}^g|_{\mR \times \p B} = \lim_{j \to \infty} u_j|_{\mR \times \p B} = U_{\omega,s}^{\gm}|_{\mR \times \p B}.
\]
Thus we have proved that the boundary values of plane waves for $g$ and $\gm$ agree. The corollary now follows from Theorem \ref{thm_rigidity}.
\end{proof}


\appendix

\section{Additional proofs} \label{sec_app}

\subsection{Eikonal equation}

In this section we give the proof of Proposition \ref{prop_eikonal}. As preparation, we first prove a version of this result that is valid in suitable open sets.

\begin{Lemma} \label{lemma_eikonal1}
Let $g$ be a smooth Lorentzian metric in $\mR^{1+n}$ with $g = \gm$ in $\{ x \cdot \omega \leq -1 \}$. Let $\theta: \Sigma_- \to \mR_+ \cup \{\infty\}$ be a lower semicontinuous function such that $\gamma_z$ is defined on $(-\infty,\theta(z))$, let $D_{\theta} = D_{\omega,\theta} = \{ (z,r) \,:\, z \in \Sigma_-, \, r < \theta(z) \}$, and let 
\[
\Phi: D_{\theta}  \to \mR^{1+n}, \ \ \Phi(z, r) = \gamma_{z}(r).
\]
The following conditions are equivalent.
\begin{enumerate}
\item[(a)] 
$\Phi: D_{\theta}  \to \Phi(D_{\theta})$ is a diffeomorphism.
\item[(b)]  
There is a smooth function $\varphi = \varphi_{\omega}$ in some open set containing $\Phi(D_{\theta})$ such that  
\[
g(d\varphi, d\varphi) = 0, \quad \varphi|_{\{x \cdot \omega \leq -1\}} = t - x \cdot \omega.
\]
\end{enumerate}
If \emph{(a)} holds, then $\varphi$ in \emph{(b)} is characterized by the property  
\[
d\varphi(\gamma_z(r)) = -\dot{\gamma}_z(r)^{\flat}, \qquad (z,r) \in D_{\theta}.
\]
In particular $\varphi$ is constant along each $\gamma_z$. Moreover, if there is another unit vector $\tilde{\omega}$ such that $g = \gm$ also in $\{ x \cdot \tilde{\omega} \leq -1 \}$, and if \emph{(b)} holds both for $\omega$ and $\tilde{\omega}$ and $d\varphi_{\omega}(z_0) = \lambda d\varphi_{\tilde{\omega}}(z_0)$ for some $z_0 \in \Phi(D_{\omega,\theta}) \cap \Phi(D_{\tilde{\omega},\tilde{\theta}})$ and some $\lambda > 0$, then $\omega = \tilde{\omega}$.
\end{Lemma}

To establish this we need Lemma \ref{lemma_bichar_geodesic}, which is proved next.

\begin{proof}[Proof of Lemma \ref{lemma_bichar_geodesic}]
The integral curve $\eta(r) = (z(r), \zeta(r))$ of $H_p$ with $\eta(0) = (\bar{z}, \bar{\zeta})$ solves the Hamilton equations 
\begin{align*}
\dot{z}^j(r) &= \p_{\zeta_j} p(\eta(r)) = -g^{jk}(z(r)) \zeta_k(r), \\
\dot{\zeta}_l(r) &= -\p_{z^l} p(\eta(r)) = \frac{1}{2} \p_{z^l} g^{km}(z(r)) \zeta_k(r) \zeta_m(r).
\end{align*}
We  denote by $\nabla$ the Levi-Civita connection for $g$. If $\zeta(r)$ is considered as a $1$-form along $z(r)$, one has 
\begin{align*}
\nabla_{\dot{z}} \zeta &= (\dot{\zeta}_l - \Gamma_{jl}^k \dot{z}^j \zeta_k) \,dx^l =  (\frac{1}{2} (\p_{z^l} g^{km}) \zeta_k \zeta_m + \Gamma_{jl}^k g^{jm} \zeta_k \zeta_m) \,dx^l \\
 &= (-\frac{1}{2} g^{ka} (\p_{z^l} g_{ab}) g^{bm} + \Gamma_{jl}^k g^{jm}) \zeta_k \zeta_m \,dx^r.
\end{align*}
After inserting the definition of the Christoffel symbol,  a short computation yields that $\nabla_{\dot{z}} \zeta = 0$, that is, $\zeta(r)$ is the parallel transport of $\bar{\zeta}$ along $z(r)$. Since $\dot{z} = -\zeta^{\sharp}$ and since $\nabla$ commutes with $\phantom{i}^{\sharp}$, this also implies that $\nabla_{\dot{z}} \dot{z} = 0$. Hence $z(r)$ is the geodesic with $z(0) = \bar{z}$ and $\dot{z}(0) = -\bar{\zeta}^{\sharp}$. Now both $\zeta(r)$ and $-\dot{z}(r)^{\flat}$ are parallel along $z(r)$ and have the same initial condition, which yields $\zeta(r) = -\dot{z}(r)^{\flat}$.

Finally, suppose that $\varphi$ is a smooth solution of $p(z,d\varphi(z)) = 0$ near $z([a,b])$ and $\zeta(r_0) = d\varphi(z(r_0))$ where $r_0 \in [a,b]$. Let $Z(r)$ be the solution of the ODE   
\begin{align*}
\dot{Z}(r) &= \nabla_{\zeta} p(Z(r),d\varphi(Z(r)))
\end{align*}
with initial condition $Z(r_0) = z(r_0)$, and write $\Xi(r) = d\varphi(Z(r))$. Then $\Xi_j(r) = \p_j \varphi(Z(r))$ and 
\[
\dot{\Xi}_j(r) = \p_{jk} \varphi(Z(r)) \dot{Z}^k(r).
\]
On the other hand, differentiating the eikonal equation $p(z,d\varphi(z)) = 0$ gives 
\[
\p_{z_j} p(z, d\varphi(z)) + \p_{\zeta_k} p(z, d\varphi(z)) \p_{jk} \varphi(z) = 0.
\]
This holds near $z([a,b])$. It follows that at least for $r$ close to $r_0$ one has 
\[
\dot{\Xi}_j(r) = \p_{jk} \varphi(Z(r)) \p_{\zeta_k} p(Z(r), d\varphi(Z(r))) = -\p_{z_j} p(Z(r), \Xi(r)). 
\]
Thus $(Z(r),\Xi(r))$ solves near $r_0$ the same Hamilton ODE system as $(z(r),\zeta(r))$ and has the same initial condition at $r_0$. By uniqueness for ODEs one has $(Z(r),\Xi(r)) = (z(r),\zeta(r))$ near $r_0$. Since the eikonal equation holds near $z([a,b])$, repeating the above argument for a larger interval gives that $\zeta(r) = d\varphi(z(r))$ near $z([a,b])$. Then 
\[
\p_r(\varphi(z(r))) = d\varphi(\dot{z}(r)) = \zeta(r) \cdot \nabla_{\zeta} p(z(r),\zeta(r)) = 0
\]
by the Euler homogeneity relation and the fact that $p(z, d\varphi) = 0$.
\end{proof}

\begin{proof}[Proof of Lemma \ref{lemma_eikonal1}]
We will write $D_{\theta} = D$ for brevity. Note that $D$ is open since $\theta$ is lower semicontinuous. First assume (a). Then (b) will follow by the geometric approach to solving eikonal equations via Lagrangian manifolds (see e.g.\ \cite[Section 6.4]{hormander} or \cite[Chapter 5]{grigis1994} for more details on the following facts). Let $\alpha = \zeta_j \,dz^j$ be the canonical $1$-form and let $\sigma = d\alpha$ be the symplectic form on $T^* \mR^{1+n}$. We let $\Lambda$ be the flowout of $\Lambda_0 = \{ (z,(1,-\omega)) \,:\, z \in \Sigma_- \}$ with respect to bicharacteristic flow of $p(z,\zeta) =  -\frac{1}{2} g_z(\zeta,\zeta)$, that is, in the notation of Lemma \ref{lemma_bichar_geodesic} 
\[
\Lambda = \{ (z(r),\zeta(r)) \,:\, z(0) \in \Sigma_-, \, \zeta(0) = (1,-\omega), \, r < \theta(z(0)) \}.
\]
Then $\Phi(D)$ is an open set by (a), and $\Lambda$ is a Lagrangian manifold in $T^*(\Phi(D)) \setminus 0$. This means that $\iota^* \sigma = 0$ where $\iota: \Lambda \to T^* (\Phi(D))$ is the natural inclusion. Thus we also have 
\[
d(\iota^* \alpha) = \iota^* (d \alpha) = \iota^* \sigma = 0.
\]

Since $\Phi: D \to \Phi(D)$ is a diffeomorphism by (a) and $D$ is contractible, also $\Phi(D)$ is contractible and its first de Rham cohomology group is trivial. Thus there is $\varphi \in C^{\infty}(\Phi(D))$ with $\iota^* \alpha = d\varphi$. Moreover, since $\Phi: D \to \Phi(D)$ is injective by (a), one has $\Lambda = \{ (z, \lambda(z)) \,:\, z \in \Phi(D) \}$ for some smooth $1$-form $\lambda$ on $\Phi(D)$. If we consider coordinates $z$ on $\mR^{1+n}$ and if $(z,\zeta)$ are associated canonical coordinates on $T^* \mR^{1+n}$, then we have  
\[
(\iota^* \alpha)_z(v) = \alpha_{(z,\lambda(z))}(\iota_* v) = \lambda_j(z) dz^j(\iota_* v) = \lambda_j(z) v^j(z) = \lambda_z(v).
\]
Thus $\lambda = \iota^* \alpha = d\varphi$, so $\Lambda = \{ (z, d\varphi(z)) \,:\, z \in \Phi(D) \}$. Since $\Lambda \subset p^{-1}(0)$ we see that $\varphi$ solves $p(z,d\varphi(z)) = 0$ in $\Phi(D)$. Finally, if $(z,r) \in D$ and if $(z(r), \zeta(r))$ is the bicharacteristic with $(z(0), \zeta(0)) = (z, (1,-\omega))$, then Lemma \ref{lemma_bichar_geodesic} gives 
\[
d \varphi(\gamma_z(r)) = d\varphi(z(r)) = \zeta(r) = -\dot{z}(r)^{\flat} = -\dot{\gamma}_{z}(r)^{\flat}.
\]

Conversely, assume that $\varphi$ is as in (b). Define the vector field $Z = -(d\varphi)^{\sharp}$, so that $Z$ is smooth near $\Phi(D)$. Lemma \ref{lemma_bichar_geodesic} gives that $\gamma_z$ is an integral curve of $Z$, that is,
\[
\dot{\gamma}_z(r) = Z(\gamma_z(r)), \qquad (z,r) \in D.
\]
Thus $\Phi(z,r) = \Psi_r(z)$ for $(z,r) \in D$ where $\Psi_r$ is the flow of $Z$. The derivative of $\Phi$ is given by 
\[
D\Phi|_{(z,r)}(\dot{z}, \dot{r}) = D\Psi_r(z)\dot{z} + Z(\Psi_r(z)) \dot{r} = D\Psi_r(z)(\dot{z} + \dot{r} Z(z)).
\]
Since $D\Psi_r$ is invertible and $Z(z)$ is transverse to $\Sigma_-$, we see that $D\Phi$ is invertible and $\Phi$ is a local diffeomorphism $D \to \Phi(D)$.

We proceed to proving injectivity of $\Phi$. By the assumption $g = \gm$ in $\{ x \cdot \omega \leq -1 \}$, any bicharacteristic $(z(r), \zeta(r))$ with $z(0) \in \Sigma_-$ and $\zeta(0) = d\varphi(z(0)) = (1,-\omega)$ stays in $\{ x \cdot \omega < -1 \}$ for negative times since the geodesic $z(r)$ is a straight line for $r < 0$. The bicharacteristic cannot return from $\{ x \cdot \omega > -1 \}$ to $\Sigma_-$ at any positive time $r$ since otherwise one would have $\zeta(r) = d\varphi(z(r)) = (1,-\omega)$ and $\dot{z}(r) = (1,\omega)$ would point into $\{ x \cdot \omega > -1 \}$, which is not possible. Thus for any $\bar{z} \in \Sigma_-$, the integral curve $\gamma_{\bar{z}}(r)$ stays in $\{ x \cdot \omega < -1 \}$ for negative times and in $\{ x \cdot \omega > -1 \}$ for positive times and it only intersects $\Sigma_-$ when $r=0$.

Now suppose that $z_1, z_2 \in \Sigma_-$ and $\Phi(z_1,r_1) = \Phi(z_2,r_2) = w$. Then the integral curve of $Z$ through $w$ contains both $z_1$ and $z_2$, so the previous discussion gives that $z_1=z_2$. If $r_1 \neq r_2$, then the integral curve contains a loop and hence is periodic, which would contradict with the discussion above. This proves that $\Phi$ is injective and therefore a diffeomorphism $D \to \Phi(D)$.

For the final statement, if $d\varphi_{\omega}(z_0) = \lambda d\varphi_{\tilde{\omega}}(z_0)$ for some $z_0$ and some $\lambda > 0$, then by following the integral curve $(z(r), \zeta(r))$ of $H_p$ through $d\varphi_{\omega}(z_0)$ backwards to $\{x \cdot \omega \leq -1\}$ or to $\{x \cdot \tilde{\omega} \leq -1\}$ (the integral curve must reach both sets since $z_0 \in \Phi(D_{\omega,\theta}) \cap \Phi(D_{\tilde{\omega},\tilde{\theta}})$) we obtain that $\dot{z}(r) = (1,\omega) = (1,\tilde{\omega})$ for $r \ll 0$. Thus $\omega = \tilde{\omega}$.
\end{proof}

We now move to the proof of Proposition \ref{prop_eikonal}. In order to study this case we employ a nontrapping condition stating that no geodesic $\gamma_z(r)$ can stay in a compact set for all $r \geq 0$. This is true under the global hyperbolicity condition \eqref{metric_a2}, since then the Cauchy temporal function must increase to $+\infty$ along such a geodesic.

\begin{proof}[Proof of Proposition \ref{prop_eikonal}]
If $\rho(z_0) > a$, then $\rho(z) > a$ for $z$ near $z_0$ by the smooth dependence on initial data of solutions to ODEs. This shows that the maximal existence time $\rho$ is lower semicontinuous, and therefore $D$ is open and connected. If (a) holds, then Lemma \ref{lemma_eikonal1} gives (b). Now assume that (b) holds. From Lemma \ref{lemma_eikonal1} we obtain that the derivative of $\Phi: D \to \mR^{1+n}$ is invertible everywhere. If we can show that $\Phi$ is a proper map, the Hadamard global inverse function theorem will imply that $\Phi$ is a diffeomorphism from $D$ onto $\mR^{1+n}$.

We now prove that $\Phi$ is proper. It is enough to consider $\omega = e_1$ and write coordinates in $\mR^{1+n}$ as $(t,x_1,y)$ where $y \in \mR^{n-1}$. Let $K \subset \mR^{1+n}$ be compact. Then there is $C > 1$ such that 
\[
K \subset \{ |t| \leq C, \, |x_1| \leq C, \, |y| \leq C \}.
\]
By \eqref{metric_a1} any geodesic $\gamma_z(r)$ for $z = (t,-1,y)$ with $|y| > C$ will stay in $\{ |y| > C \}$. Thus 
\[
\Phi^{-1}(K) \subset \{  |y| \leq C \}.
\]
Since $\varphi$ is continuous, there is $T > 0$ such that $|\varphi| \leq T-1$ in $K$. Now if $(z,r) \in \Phi^{-1}(K)$ where $z = (t,-1,y)$, then the fact that $\varphi$ is constant along $\gamma_z$ gives that 
\[
|t+1| = |\varphi(z)| = |\varphi(\gamma_z(r))| \leq T-1.
\]
This yields 
\[
\Phi^{-1}(K) \subset \{ (z,r) \,:\, z \in S \},
\]
where $S$ is the compact set 
\[
S = \{ (t,-1,y) \in \Sigma_- \,:\, |t| \leq T, \, |y| \leq C \}.
\]

Finally we use the assumption \eqref{metric_a2}, which ensures the existence of a smooth Cauchy temporal function $\tau$. Since $K$ is compact, there is $L > 0$ such that $|\tau| \leq L$ in $K$. Now $\tau$ is strictly increasing along the inextendible causal curve $\gamma_z$, so there is a unique number $R(z) < \rho(z)$ with $\tau(\gamma_z(R(z))) = L$. Thus in particular $\gamma_z(r) \notin K$ when $r > R(z)$. We claim the following uniform nontrapping statement: 
\begin{equation} \label{nontrapping_uniform}
\sup_{z \in S} R(z) < \infty.
\end{equation}
We argue by contradiction and suppose that for any $j$ there is $z_j \in S$ with $R(z_j) > j$. After passing to a subsequence, we may assume that $z_j \to z \in S$ and $R(z_j) \to \infty$. Choose $r_0 > R(z)$, so that $\tau(\gamma_z(r_0)) > L$ and also $\tau(\gamma_{z_j}(r_0)) > L$ for $j$ sufficiently large since $z \mapsto \gamma_z(r_0)$ is continuous. Thus $R(z_j) < r_0$ for $j$ large, which contradicts the fact that $R(z_j) \to \infty$. This proves \eqref{nontrapping_uniform}.

Combining the above facts and writing $R = \sup_{z \in S} R(z)$ proves that 
\[
\Phi^{-1}(K) \subset \{ (z,r) \,:\, z \in S, \, r \leq R \}.
\]
By \eqref{metric_a1} we also have $\Phi^{-1}(K) \subset \{ (z,r) \,:\, r \geq -C+1 \}$. Thus $\Phi^{-1}(K)$ is a closed subset of $D$ that is contained in a compact subset of $\Sigma_- \times \mR$. This concludes the proof that $\Phi^{-1}(K)$ is compact. Hence $\Phi$ is proper and thus it is a diffeomorphism $D \to \mR^{1+n}$. This proves that (b) implies (a).

It remains to show that if (a) holds, then $\varphi = t-x_1$ outside $\mR \times B$ and that $D = \Sigma_- \times \mR$. Since $\varphi = t-x_1$ in $\{ x_1 \leq -1 \}$,  the fact that $\varphi$ is constant along each $\gamma_z$ and \eqref{metric_a1} imply that 
\[
\varphi = t-x_1 \text{ in $\{ x_1 \leq -1 \} \cup \{ |y| \geq 1 \}$}.
\]
Now let $w \in \Sigma_+ \cap \{ |y| \leq 1 \}$. Since $\Phi$ is a diffeomorphism onto $\mR^{1+n}$ there is a unique $z \in \Sigma_-$ such that $w = \gamma_z(r_0)$ for some $r_0$, and by \eqref{metric_a1} one must have $z \in \{ |y| \leq 1 \}$.

Let us study the tangent vector of $\gamma_z$ at time $r_0$, written as 
\[
\dot{\gamma}_z(r_0) = (\lambda_0, \lambda_1, v)
\]
where $\lambda_0, \lambda_1 \in \mR$ and $v \in \mR^{n-1}$. Here $\lambda_0 > 0$ since $\gamma_z$ is future-directed. We also have $\lambda_1 \geq 0$, since otherwise $\gamma_z$ would have come from $\{ x_1 > 1 \}$ before time $r_0$, and then by \eqref{metric_a1} $\gamma_z(r)$ would have been a line segment in $\{ x_1 > 1 \}$ for $r < r_0$, which is not possible. Since $\lambda_1 \geq 0$,  \eqref{metric_a1} implies that $\gamma_z(r)$ for $r \geq r_0$ is a line segment in $\{ x_1 \geq 1 \}$ with constant tangent vector $(\lambda_0, \lambda_1, v)$. Now if $v \neq 0$, then $\gamma_z$ would reach some point $\tilde{w} \in \{ |y| > 1 \}$ in finite time. But one also has $\tilde{w} = \gamma_{\tilde{z}}(\tilde{r})$ for some $\tilde{z} \in \Sigma_- \cap \{ |y| > 1 \}$ by \eqref{metric_a1}. This contradicts the fact that $\Phi$ is a diffeomorphism onto $\mR^{1+n}$, so one must have $v = 0$. The fact that $\dot{\gamma}_z(r_0)$ is null then gives that $\lambda_0 = \lambda_1 > 0$.

We have proved that any $w \in \Sigma_+ \cap \{ |y| \leq 1 \}$ is of the form $w = \gamma_z(r_0)$ for some $z \in \Sigma_- \cap \{ |y| \leq 1 \}$, and there is $\lambda > 0$ such that 
\[
\dot{\gamma}_z(r_0) = \lambda(1, e_1).
\]
Now Lemma \ref{lemma_eikonal1} gives that 
\[
d\varphi(w) = -\dot{\gamma}_z(r_0)^{\flat} = \lambda(1, -e_1).
\]
In particular $\nabla_y \varphi(w) = 0$ for any $w \in \Sigma_+ \cap \{ |y| \leq 1 \}$. Since $\varphi = t-x_1$ in $\{ |y| \geq 1 \}$, it follows that $\varphi = t-x_1$ on $\Sigma_+$. Moreover, $\lambda = \p_t \varphi = 1$ on $\Sigma_+$. Using that $\varphi$ is constant along the curves $\gamma_z$, we obtain $\varphi = t-x_1$ in $\{ x_1 \geq 1 \}$. By \eqref{metric_a1} we finally get the  property that 
\[
\varphi = t-x_1 \text{ outside $\mR \times B$.}
\]

It remains to show that $D = \Sigma_- \times \mR$. Lemma \ref{lemma_eikonal1} gives that 
\[
\dot{\gamma}_z(r) = -d\varphi(\gamma_z(r))^{\sharp} = (1,e_1), \qquad \gamma_z(r) \notin \mR \times B.
\]
Thus by \eqref{metric_a1}, any $\gamma_z(r)$ that reaches $\Sigma_+$ at time $r_+(z)$ must stay in $\{ x_1 > 1 \}$ for $r > r_+(z)$. The fact that $\Phi: D \to \mR^{1+n}$ is a diffeomorphism ensures that any $w \in \Sigma_+$ lies on a unique $\gamma_z$. Since $\dot{\gamma}_z(0) = \dot{\gamma}_z(r_+(z)) = (1,e_1)$, the family $\gamma_z$ satisfies the properties in Lemma \ref{lemma_kappa}, and therefore the map $\kappa: \Sigma_+ \to \Sigma_-$, $w \mapsto z$, is an embedding.

Next we prove that $\kappa$ is surjective. If $w \in \Sigma_{+,\sigma}$ for some $\sigma \in \mR$ and if $z = (t, -1, \tilde{y}) = \kappa(w)$, then the fact that $\varphi$ is constant along $\gamma_z$, together with $\varphi = t-x_1$ for $|x_1| \geq 1$, gives 
\[
\sigma - 1 = \varphi(w) = \varphi(z) = t + 1.
\]
It follows that $\kappa$ maps $\Sigma_{+,\sigma}$ to $\Sigma_{+,\sigma-2}$. Since $\kappa$ is an embedding, $\kappa(\Sigma_{+,\sigma})$ is open in $\Sigma_{+,\sigma-2}$. By \eqref{metric_a1} one has $\kappa(\Sigma_{+,\sigma} \cap \{ |y| \geq 1 \}) = \Sigma_{-,\sigma-2} \cap \{ |y| \geq 1 \}$. On the other hand $\kappa(\Sigma_{+,\sigma} \cap \{ |y| \leq 1 \})$ is compact and hence closed. Combining these facts gives that $\kappa(\Sigma_{+,\sigma})$ is both open and closed, so by connectedness $\kappa(\Sigma_{+,\sigma}) = \Sigma_{-,\sigma-2}$. Thus $\kappa: \Sigma_+ \to \Sigma_-$ is a surjective embedding, hence a diffeomorphism, and any $\gamma_z$ reaches $\Sigma_+$ at some time $r_+(z)$ with $\dot{\gamma}_z(r_+(z)) = (1,e_1)$. It then follows from \eqref{metric_a1} that $D = \Sigma_- \times \mR$.
\end{proof}

\subsection{Plane waves}

Here we give a proof of Lemma \ref{lem_pw_conormal}.

\begin{proof}[Proof of Lemma \ref{lem_pw_conormal}]
Write $\beta = \varphi - s$. We use the progressing wave expansion and consider an approximate plane wave 
\[
U^{(N)} = \sum_{j=0}^{N} u_j H_j(\beta)
\]
where $H_j(r) = r^j H(r)$ and $u_j$ are smooth functions in $\mR^{1+n}$. Recalling the sign convention for $\Box_g$, for any smooth function $v$ we have 
\begin{equation} \label{boxg_product}
\Box_g(v H_j(\beta)) = -g(d\beta,d\beta) v H_j''(\beta) + (L v) H_j'(\beta) + (\Box_g v) H_j(\beta)
\end{equation}
where $L$ is the first order operator (with $Z$ as in Lemma \ref{lemma_eikonal}) 
\[
Lv = 2 Zv + (\Box_g \beta) v.
\]
Therefore also 
\begin{align*}
\Box_g U^{(N)} = \sum_{j=0}^{N} \left[ -g(d\beta,d\beta) u_j H_j''(\beta) + (L u_j) H_j'(\beta) + (\Box_g u_j) H_j(\beta) \right].
\end{align*}
Since $g(d\beta,d\beta) = 0$, this reduces to 
\begin{align*}
\Box_g U^{(N)} = (L u_0) \delta(\beta) + \sum_{j=1}^{N} \left[ j L u_j + (\Box_g u_{j-1}) \right] H_{j-1}(\beta) + (\Box_g u_{N}) H_{N}(\beta).
\end{align*}

The next step is to choose the functions $u_j$ as solutions of the transport equations 
\[
L u_0 = 0, \qquad u_0|_{\{ x \cdot \omega \leq -1 \}} = 1,
\]
and for $j \geq 1$ 
\[
j L u_j + (\Box_g u_{j-1}) = 0, \qquad u_j|_{\{ x \cdot \omega \leq -1 \}} = 0.
\]
Recall from Lemma \ref{lemma_eikonal} that the integral curves of $Z$ are the null geodesics $\gamma_z(r)$. The diffeomorphism property in Proposition \ref{prop_eikonal} ensures that no integral curve of $Z$ is trapped in a compact set, and this implies that the transport equations have smooth solutions $u_j$ for any $j \geq 1$. More concretely, any $w \in \mR^{1+n}$ can be written uniquely as $w = \Phi(z,r) = \gamma_z(r)$ where $z = z(w)$ and $r = r(w)$ depend smoothly on $w$. Then $u_0$ is given by 
\begin{align*}
u_0(w) = u_0(\gamma_z(r)) &= \exp \left[ -\frac{1}{2} \int_0^r (\Box_g \beta)(\gamma_z(\rho)) \,d\rho \right],
\end{align*}
and for $j \geq 1$ one has 
\begin{align*}
u_j(\gamma_z(r)) &= -\frac{1}{2j} e^{-\frac{1}{2} \int_0^r \Box_g \beta(\gamma_z(\rho)) \,d\rho} \int_0^r (\Box_g u_{j-1})(\gamma_z(\rho)) e^{\frac{1}{2} \int_0^{\rho} (\Box_g \beta)(\gamma_z(s)) \,ds} \,d\rho.
\end{align*}

With the above choices, one has 
\[
\Box_g U^{(N)} = (\Box_g u_{N}) H_{N}(\beta) \in H^N_{\mathrm{loc}}(\mR^{1+n}), \qquad U^{(N)}|_{\{ x \cdot \omega \leq -1 \}} = H(t-x \cdot \omega - s).
\]
We choose $\tau_0$ so that \eqref{uzero_more_precise_bound} holds, and take $R^{(N)}$ to be the forward solution of 
\[
\Box_g R^{(N)} = -\Box_g U^{(N)}, \qquad R^{(N)}|_{\{ \tau < \tau_0 \}} = 0.
\]
Proposition \ref{prop_wp} gives that $R^{(N)} \in H^{N-1}_{\mathrm{loc}}(\mR^{1+n})$. We wish to prove that for any $N$, the plane wave $U$ is given by 
\[
U = U^{(N)} + R^{(N).}
\]
By the uniqueness part of Proposition \ref{prop_pw}, this follows if we can show that 
\begin{equation} \label{un_values_past}
U^{(N)}|_{\{ \tau < \tau_0\}} = H(t-x \cdot \omega -s)|_{\{ \tau < \tau_0 \}}.
\end{equation}
Writing $\omega = e_1$ and $x = (x_1, x')$, condition \eqref{metric_a1} ensures that $L = \p_t + \p_{x_1}$ in  $\{ x_1 \leq -1 \text{ or } |x'| \geq 1 \}$. Thus we have 
\[
u_0 = 1 \text{ and } u_j = 0 \text{ in } \{ x_1 \leq - 1 \text{ or } |x'| \geq 1 \} \text{ when $j \geq 1$.}
\]
In particular $U^{(N)} = H(t-x\cdot \omega - s)$ in $ \{ x_1 \leq - 1 \text{ or } |x'| \geq 1 \}$. On the other hand, we claim that 
\begin{equation} \label{un_varphi}
\{ \tau \leq \tau_0 \} \cap \{ x_1 \geq -1, \, |x'| \leq 1 \} \subset \{ \varphi < s \}.
\end{equation}
Since $U^{(N)} = 0$ in $\{ \varphi < s \}$ and $H(t-x \cdot \omega -s) = 0$ in $\{ \tau \leq \tau_0 \} \cap \{ x_1 \geq -1, \, |x'| \leq 1 \}$ by \eqref{uzero_more_precise_bound}, this would imply \eqref{un_values_past}. To show \eqref{un_varphi}, fix $w \in \{ \tau \leq \tau_0 \} \cap \{ x_1 \geq -1, \, |x'| \leq 1 \}$ and use the diffeomorphism property in Proposition \ref{prop_eikonal} to write $w = \gamma_z(r)$ for some $z \in \Sigma_-$ and $r \geq 0$. Note that $z \in \{ |x'| \leq 1 \}$ by \eqref{metric_a1}. Now $\tau$ is increasing along $\gamma_z$, so 
\[
\tau(z) \leq \tau(w) \leq \tau_0.
\]
This implies that $z \in \{ t < s-1, \, |x'| \leq 1 \}$ by Proposition \ref{prop_tau_proper} (a) if we chose $\tau_0$ small enough to begin with. The function $\varphi$ is constant along $\gamma_z$, which gives  $\varphi(w) = \varphi(z) < s$ and proves \eqref{un_varphi}.

To show that $U \in I^{-\frac{1+n}{4}}(\mR^{1+n}, Y)$, we use \cite[Definition 18.2.6]{hormander} and fix $N \geq 1$ together with first order operators $L_1, \ldots, L_N$ with smooth coefficients tangential to $Y$. We proved above that 
\begin{equation} \label{pw_rep_nplusone}
U = U^{(N+1)} + R^{(N+1)}.
\end{equation}
We need to show that 
\[
L_1 \ldots L_N U \in L^2_{\mathrm{loc}}(\mR^{1+n}).
\]
Since $R^{(N+1)} \in H^{N}_{\mathrm{loc}}$, we have $L_1 \ldots L_N R^{(N+1)} \in L^2_{\mathrm{loc}}$. On the other hand the function $U^{(N+1)}$ is smooth away from $Y$, so by \eqref{pw_rep_nplusone} it is enough to show that any $z_0 \in Y$ has a neighborhood $V$ in $\mR^{1+n}$ such that 
\[
L_1 \ldots L_N U^{(N+1)}|_V \in L^2(V).
\]
Fix $z_0 \in Y$ and choose $V$ so small that there are local coordinates $y = (y_0, \ldots, y_n)$ in $V$ with $\beta = y_0$ in $V$. Then in $V$ one has 
\[
U^{(N+1)} = \sum_{j=0}^{N+1} u_j(y) H_j(y_0)
\]
where $u_j$ are smooth in $V$. By \cite[Lemma 18.2.5]{hormander}, in $V$ each operator $L_j$ has the form 
\[
a_0(y) y_0 \p_{y_0} + \sum_{j=1}^n a_j(y) \p_{y_j} + b(y)
\]
where $a_j$ and $b$ are smooth. It follows that $L_1 \ldots L_N U^{(N+1)}|_V \in L^2(V)$ as required, and therefore $U \in I^{-\frac{1+n}{4}}(\mR^{1+n}, Y)$.

Finally, to show that $U = u H(\beta)$ where $u$ is smooth, we note that $U|_{\{ \beta \geq 0 \}} \in H^N_{\mathrm{loc}}(\{ \beta \geq 0 \})$ by \eqref{pw_rep_nplusone}. Since this is true for any $N$ and since $U$ is supported in $\{ \beta \geq 0 \}$, the required representation follows.
\end{proof}

\bibliographystyle{alpha}
\bibliography{master}

\end{document}